\documentclass[11pt, reqno]{amsart}

\usepackage{caption, amssymb,amsfonts,amsmath, amsthm, graphicx, geometry}
\usepackage{pgf,tikz}
\usepackage{pgfplots}
\usepackage[nice]{nicefrac}
\usepackage{mathrsfs}
\usepackage[colorlinks=true,linkcolor = blue, citecolor= blue , bookmarksnumbered=true, pdfauthor={Jacob, M\"orters, Kerriou, Linker}]{hyperref}
\usepackage{diagbox}
\usepackage{comment}
\usepackage{hhline}
\usepackage{enumerate}
\usepackage[normalem]{ulem}
\usepackage{stmaryrd}
\numberwithin{equation}{section}
\usepackage{array}
\usepackage[normalem]{ulem} 
\usepackage{graphicx}
\usepackage{nicematrix}
\usepackage{enumitem}

\graphicspath{ {./images/} }

\usetikzlibrary{arrows,decorations.markings}

\definecolor{miazul}{RGB}{0,197,253}
\definecolor{ultramarine}{RGB}{40,72,140}
\definecolor{purple}{RGB}{156,75,119}
\definecolor{miyellow}{RGB}{190,190,30}
\definecolor{salmon}{RGB}{235,126,78}
\definecolor{naranja}{RGB}{250,83,0}
\definecolor{migreen}{RGB}{49,156,85}
\definecolor{darkred}{RGB}{145,6,0}

\newlength{\hatchspread}
\newlength{\hatchthickness}
\newlength{\hatchshift}
\newcommand{\hatchcolor}{}
\tikzset{hatchspread/.code={\setlength{\hatchspread}{#1}},
	hatchthickness/.code={\setlength{\hatchthickness}{#1}},
	hatchshift/.code={\setlength{\hatchshift}{#1}},
	hatchcolor/.code={\renewcommand{\hatchcolor}{#1}}}
\tikzset{hatchspread=3pt,
	hatchthickness=0.4pt,
	hatchshift=0pt,
	hatchcolor=black}
\pgfdeclarepatternformonly[\hatchspread,\hatchthickness,\hatchshift,\hatchcolor]
{custom north west lines}
{\pgfqpoint{\dimexpr-2\hatchthickness}{\dimexpr-2\hatchthickness}}
{\pgfqpoint{\dimexpr\hatchspread+2\hatchthickness}{\dimexpr\hatchspread+2\hatchthickness}}
{\pgfqpoint{\dimexpr\hatchspread}{\dimexpr\hatchspread}}
{
	\pgfsetlinewidth{\hatchthickness}
	\pgfpathmoveto{\pgfqpoint{0pt}{\dimexpr\hatchspread+\hatchshift}}
	\pgfpathlineto{\pgfqpoint{\dimexpr\hatchspread+0.15pt+\hatchshift}{-0.15pt}}
	\ifdim \hatchshift > 0pt
	\pgfpathmoveto{\pgfqpoint{0pt}{\hatchshift}}
	\pgfpathlineto{\pgfqpoint{\dimexpr0.15pt+\hatchshift}{-0.15pt}}
	\fi
	\pgfsetstrokecolor{\hatchcolor}
	\pgfusepath{stroke}
}

\usetikzlibrary{calc}
\usepackage{relsize}

\tikzset{fontscale/.style = {font=\relsize{#1}}
}
\usepackage{smartdiagram}

\newcommand{\eps}{\varepsilon}

\newcommand{\R}{\mathbb R}

\newcommand{\E}{\mathbb E}
\renewcommand{\P}{\mathbb P}

\newcommand{\poisson}{\mathcal{P}}

\newcommand{\simone}{\overset {1-\text{edge}}{\longleftrightarrow} }
\newcommand{\simtwo}{\overset {2-\text{edge}}{\longleftrightarrow} }
\newcommand{\simk}{\overset {k-\text{edge}}{\longleftrightarrow} }
\newcommand{\usl}{\underset{^{_{\log}}}\asymp}
\newcommand{\usll}{\underset{^{_{\log}}}\lesssim}

\renewcommand{\complement}{{\sf c}}

\renewcommand{\phi}{\varphi}

\newcommand{\N}{\mathbb N}

\renewcommand{\P}{\mathbb P}
\newcommand{\vertiii}[1]{{\left\vert\kern-0.25ex\left\vert\kern-0.25ex\left\vert #1 
    \right\vert\kern-0.25ex\right\vert\kern-0.25ex\right\vert}}

\newcommand\cD{\mathcal D}
\newcommand\cE{\mathcal E}

\newcommand\cP{\mathcal P}

\newcommand\cS{{\mathcal S}}

\newcommand{\de}{\mathrm{d}} 

\newcommand{\x}{\mathbf{x}}
\newcommand{\y}{\mathbf{y}}
\newcommand{\z}{\mathbf{z}}

\newcommand{\network}{\mathscr{G}}

\numberwithin{equation}{section}

\newtheorem{theorem}{Theorem}
\newtheorem{shared}{Dummy}[section]  

\newtheorem{definition}[shared]{Definition}
\newtheorem{lemma}[shared]{Lemma}
\newtheorem{corollary}[shared]{Corollary}
\newtheorem{proposition}[shared]{Proposition}
\newtheorem{remark}[shared]{Remark}
\newtheorem{conjecture}[shared]{Conjecture}

\title[Crossing probabilities in geometric  inhomogeneous random graphs]
{Crossing probabilities in  geometric inhomogeneous random graphs}

\author[Emmanuel Jacob, Céline Kerriou, Amitai Linker and Peter M\"orters]{Emmanuel Jacob, Céline Kerriou, Amitai Linker and Peter M\"orters}

\pgfplotsset{compat=1.18}

\begin{document}
\maketitle

\vspace{-.8cm}
\medskip

\begin{quote}
{\small {\bf Abstract:} {In  a geometric inhomogeneous random graph vertices are given by the points of a Poisson process and are equipped with independent weights following a heavy tailed distribution.  Any pair of distinct vertices is independently forming an edge  with a probability decaying as a function of the product of the weights divided by the distance of the vertices. For this continuum percolation model we study the crossing probabilities of annuli, i.e.\ the probabilities that there exist paths starting inside a ball and ending outside a larger concentric ball with increasing inner and outer radii.
Depending on the radii, 
the power-law exponent of the degree distribution and the decay of the probability of long edges, we identify regimes where the crossing probabilities by a path are equivalent to the crossing probabilities by one or by two edges.
We also identify the escape probabilities from balls with strong centre, i.e.\ the asymptotics of the probability that there exists a path starting from a vertex with a given weight leaving a centred ball as radius and weight are going to infinity. 
As a corollary we get the subcritical one-arm exponents characterising the decay of the probability that a typical point is in a component not contained in a centred ball whose radius goes to infinity.}}
\end{quote}

\tableofcontents
\vspace{-0.5cm}
\section{Introduction and main results}

\subsection{Motivation}
In general continuum percolation models, vertices are given by the points of a homogeneous Poisson point process on $\mathbb R^d$
and edges are formed randomly according to some translation invariant probability law leading to finite expected vertex degrees.
Phase transitions in these models express themselves in a sudden qualitatively change of the
graph topology when the intensity of the underlying point process crosses a positive and finite critical value. This change of behaviour arises not only in the dichotomy of existence of an infinite, dense connected component for supercritical densities versus only finite components for subcritical densities, but also in various other forms - such as the existence or absence of a path comprising many different edges, or the existence or absence of a path
crossing an annulus of large width. The latter is of particular interest because of its local and quantitative nature. If the crossing probability of an annulus with fixed ratio of inner and outer radius goes to zero as the width goes to infinity, then this fact can be used to initiate  renormalization arguments, directly showing many features characterising subcritical behaviour - such as finite expected component size and exponential decay of the probability that there is a path of length~$n$ starting at a typical point, see Gou\'er\'e and Th\'eret~\cite{G08, GT19}. We hence speak 
of the \emph{quantitatively subcritical phase} of the percolation model.
\smallskip

In the classical continuum percolation model of Gilbert~\cite{Gil} vertices are connected by an edge if their distance lies below a fixed threshold value. In this model all phase transitions occur at the same critical intensity and the transitions are sharp. In particular, 
for subcritical densities the probability that there is a path from a point inside to a point outside of an  annulus decays exponentially in the width of the annulus, while
for supercritical  densities the number of vertex-disjoint paths crossing the annulus grows at least like a power of the width~\cite{tanemura, faggionato}, see \cite{armstrong2023} for related analytic results. At criticality, crossing probabilities in two-dimensional percolation have rich relations to conformal field theory and stochastic Loewner evolution and have been central to progress in percolation theory in the last twenty years, see for example~\cite{Cardy, Schramm} for lattice models and \cite{Mertens} for nonrigorous results on continuum percolation. 
\smallskip%

For continuum percolation models with sufficiently heavy-tailed edge length distribution or with a power-law degree distribution one cannot expect exponential decay of the annulus crossing probabilities for subcritical densities. In the former case 
with only polynomially decaying probability a single edge can cross the annulus and in the latter case a single strong vertex inside the annulus can be adjacent to an edge connecting it to the 
outside, or vice versa. Therefore crossings using only one edge give polynomial lower bounds for the subcritical crossing probabilities. In~\cite{AHL_2020__3__677_0} Duminil-Copin et al study the Poisson-Boolean model, a percolation model which has 
power-law degree distributions.
Under a moment assumption they show that for subcritical densities, the crossing probability of an annulus with fixed ratio of inner and outer radius is asymptotically equivalent to the probability that a single edge crosses the annulus. See also the recent work of Dembin and Tassion~\cite{Dembin2022}, where the moment condition is relaxed.
\smallskip

In this paper we consider a different long-range percolation model often called the \emph{geometric inhomogeneous random graph}. In this model,  vertices are placed according to a homogeneous Poisson point process in space and each vertex carries an independent weight. Given vertices and weights, edges are created independently between vertices with a probability given as a function 
depending on the product of the weights of the adjacent vertices and their distance in space. 
Choosing heavy-tailed weight distributions and a
slowly decreasing profile leads to a power-law degree distribution and heavy-tailed edge length distribution.
We study this model in the quantitatively subcritical phase.
\smallskip

\pagebreak[3]

Recent work of Jacob et al \cite{jacob2025} gives criteria for existence and absence of a quantitatively subcritical phase. It turns out that when the probability that an annulus with proportional inner and outer radius is crossed by a single edge does not go to zero, this is not enough to guarantee the presence of an infinite component. In this case there is a subcritical, but no quantitatively subcritical phase. Conversely, if there is a quantitatively subcritical phase, the probability that such annuli are crossed by a single edge goes to zero with a rate that does not depend on the intensity of the underlying Poisson process. In this case the natural conjecture that the subcritical and quantitatively subcritical phase agree is still an open problem, which we do not address here. 
\smallskip

Here we  focus on the quantitatively subcritical phase and give
up-to-constants matching bounds for the crossing probabilities
of annuli with outer radius growing faster than the inner radius. These up-to-constants bounds are of polynomial order in the radii and do not depend on the intensity of the underlying Poisson process
as long as it is quantitatively subcritical.
In particular we show that, if the tail of the long edge probability is heavy relative to the tail of the degree distribution, the order of the crossing probability is that of the crossing by a single edge, but if it is light, then 
it depends on the 
rate of divergence of the inner radius compared to the outer radius whether we see the one-edge or two-edge asymptotics. Moreover, depending on the power-law exponent of the degree distributions the order of the crossing probabilities may or may not depend on the  tail of the long edge probabilities. We also obtain up to constants bounds for the escape probabilities from balls with a strong centre, 
i.e.\ the probability that there exists a path from a vertex to the outside of a ball around that vertex as the weight of the centre and the radius of the ball go to infinity. As before we see different parameter regimes in which paths with one or two edges determine the escape probabilities. Finally, as a corollary of our main results, we get the subcritical one-arm exponents, which describe the polynomial rate at which the probability that a typical point is in a component not contained in a centred ball goes to zero as the radius goes to infinity. Again, these exponents do not depend on the underlying intensity as long as it is quantitatively subcritical.

\subsection{Background and description of the model }
\label{sec:background}
We consider the \emph{geometric inhomogeneous random graph}
\cite{BRINGMANN201935}, also known as weight-dependent random connection model~\cite{GLM21, Grac} or kernel-based spatial random graph~\cite{JKM24, JKM}  with factor kernel, or as continuum scale-free percolation model \cite{DHH13, Dep} in the literature.
In this model, vertices are given by a homogeneous Poisson point process $\poisson$ on $\R^d\times(0,1]$ with intensity $\lambda>0$,  which represents the density of vertices in the model. We interpret a point $\x =(x,u)\in \poisson$  of this Poisson point process as having position $x\in \R^d$ and mark $u$ uniform in $(0,1]$. Whenever we want to emphasize the role of $\lambda$ we denote the underlying probability measure by \smash{$\mathbb P^{(\lambda)}$}.
\smallskip

The random graph $\network = (\poisson, E)$ has  vertex set $\poisson$ and, given the vertex set, the edge set $E$ is defined such that for every pair of vertices  $(x,u),(y,v)$ we sample the edge $((x,u), (y,v))$, denoted by $(x,u)\sim (y,v)$, independently with probability 
\begin{align*}
\rho(|y-x|^d(uv)^{\gamma}),
\end{align*}
where $0\leq\gamma<1$ 
is a parameter regulating the inhomogeneity of the model and $\rho\colon [0,\infty)\to [0,1]$ is a non-increasing \emph{profile function} satisfying
\begin{align*}
0<	\int_{\R^d} \rho(|x|^d)\, dx <\infty.
\end{align*}
In this paper, we consider both hard and soft profile functions. More precisely, in the hard model we choose $\rho$ of the form	$\rho= 1_{[0,r_0]}$, for some positive constant $r_0$, in which case an edge $(x,u) \sim (y,v)$ exists if and only if 
\[
|y-x|^d\le r_0 u^{-\gamma}v^{-\gamma}.
\]
In the soft model,  we choose $\rho$  such that there are constants $0<{\mathtt c}, {\mathtt C}<\infty$ such that 
\begin{align}\label{eq:conditions_rho}
    {\mathtt c}x^{-\delta} \leq  \rho(x) \leq {\mathtt C} x^{-\delta} \quad\text{  for all  }x\ge1,
\end{align}
for a \emph{decay exponent} $\delta>1$.
We characterise the hard model by formally setting the  decay exponent to~$\delta=\infty$. In the case $\gamma=0$ we formally set $1/\gamma=\infty$ and only consider the soft model, which then becomes the continuum long-range percolation or random connection model~\cite{Pen, burton}. Although our results are also new in this case, our main focus in this paper is on the case $\gamma>0$, where the model is inhomogeneous with a degree distribution of power-law type with exponent \smash{$\tau=1+\frac1\gamma$}.
An alternative interpretation of the very same model is that every vertex $(x,u)$ gets an independent Pareto-distributed weight $W_x=u^{-\gamma}$ and vertices in location $x,y$ are connected with probability
\smash{$\rho(\!\!$ {\scriptsize $\frac{|x-y|^d}{W_xW_y}$} $\!\!)$}.
\medskip

We write $B(x,r) \subset \R^d$ to denote the open ball of radius $r>0$ centred at $x$ and $B(x,r)^{\complement} = \R^d \setminus B(x,r)$ to denote its complement. 
Given two disjoint sets $A,B \subset \R^d$ we write 
$A \longleftrightarrow B$ if there exists a path 
with endpoints in $A$, resp.~$B$. 
Here and throughout the paper paths are assumed to be self-avoiding.
Given $k\in \N$ we write
\[A^{_{\, \simk}} B,\]
if there is such a path with at most $k$ edges. In particular, $B(x,r)^{_{\simone}} B(x,r')^\complement$ means that exists an edge with endpoints in $B(x,r)$ and $B(x,r')^{\complement}$. 
Using this notation we define the probabilities of crossing an annulus,
\begin{align*}
\theta_r&:=\theta_r(\lambda)=\P^{(\lambda)}\big(B(0,r)\longleftrightarrow  B(0,2r)^\complement\big), 
\\ \pi_r&:=\pi_r(\lambda)=\P^{(\lambda)}\big(B(0,r)^{_{\simone}} B(0,2r)^\complement\big),
\end{align*}
for $r>0$. For notational simplicity, we will write the dependence in $\lambda$ only when necessary. Next, we define the critical parameter
\begin{align*}
	\widehat{\lambda}_c:= \sup\{\lambda\ge 0, \lim_{r\to \infty}\theta_r(\lambda)= 0\}.
\end{align*}
The regime $0<\lambda<\widehat{\lambda}_c$ is the \emph{strongly} or \emph{quantitatively subcritical phase} when, loosely speaking, looking at large scale (i.e.\ a ball with radius of order $r$) it is unlikely that there exists a component of macroscopic size (i.e.\ a component of diameter larger than $r$) in this~ball.
Of course, \smash{$\widehat{\lambda}_c$} is 
{no larger} 
than the classical critical parameter 
\begin{align*}
	\lambda_c:=\sup\big\{\lambda\ge 0: \, \lim_{r\to \infty} \P^{(\lambda)}\big(B(0,1)\longleftrightarrow B(0,r)^\complement\big)= 0\big\}.
\end{align*}
In \cite[Section 2]{jacob2025} the authors consider a general model that contains our model when considering the same parameters $\delta$ and $\gamma$ as we do, and $\alpha=\gamma$ as their third parameter. We now summarize their results in our context. They show in Theorem 2.1 that the behaviour of $\theta$ is governed by an exponent $\zeta$ that we get from their Proposition 2.3: 
\begin{enumerate}
    \item If $\delta>2$ and $\gamma<1/2$, then
    $\displaystyle
        \zeta=\max\left\{2-\delta,2-\frac 1 \gamma\right\}<0$.\\[1mm]
    \item If  $\delta<2$ or $\gamma>1/2$, then
    $\displaystyle
    \zeta=\max\left\{2-\delta,\frac {2\gamma-1}{2\gamma-1/\delta}\right\}>0.
    $\\[2pt]
    
    \item Otherwise, namely if $\delta=2$ and $\gamma\le 1/2$, or $\gamma=1/2$ and $\delta\ge 2$, then $\zeta=0$.
\end{enumerate}
\smallskip

\noindent
In the case $\zeta>0$, we have $\widehat \lambda_c=0$ and there is no quantitatively subcritical phase. We will show in Lemma~\ref{lem:quant} that in our model this also applies if 
$\zeta=0$, i.e.\  when
$\gamma=\frac12$ or $\delta=2$. In the case $\zeta<0$, we have \smash{$\widehat\lambda_c>0$}, still from their Theorem 2.1,  as well as 
\[
\limsup_{r\to\infty} \frac {\log \theta_r}{\log r}=d \zeta, \quad  
\text{ for all } 0<\lambda<\widehat\lambda_c.  
\]
By carefully looking at their proof, we can see that the limsup is a true limit, and thus 
\[
\theta_r=r^{d \zeta (1+o(1))} \text{ as } r\to \infty, \quad  
\text{ for all } 0<\lambda<\widehat\lambda_c.
\]
In this paper, we are interested in the quantitatively subcritical phase, so we always assume that $\delta>2$ and $\gamma<1/2$ in order to have \smash{$\widehat \lambda_c>0$}. Moreover, we will mainly study the regime \smash{$0<\lambda<\widehat \lambda_c$}. In this context, our aim is to provide more precise bounds on the crossing probabilities, not only of an annulus with outer and inner radius of a fixed ratio, as in the definition of~$\theta_r$, but also of annuli where the ratio between outer and inner radius goes to infinity. For such annuli we ask whether crossings by one, two or perhaps a larger number of edges are still asymptotically equally  probable as crossings with  arbitrarily long paths.
\medskip

\paragraph{\textbf{Notation}}
We reserve the word \emph{domain} for Borel subsets
\begin{align} \label{def:domain_D}
    D \subseteq \mathcal D:=\{ (u,r, R) \in (0,1]\times [2,\infty) \times [4,\infty) \colon  R\ge 2r\}.
\end{align}
Domains are characterised by stating the additional restrictions on $(u,r,R)\in\mathcal D$.
If there are no additional restrictions on a variable or a variable is not used in the given context we may omit it from the notation, writing for example $(r,R)\in\{ r> R^\epsilon\}$ if
$ (r, R) \in [2,\infty) \times [4,\infty)$ satisfy $  R/2 \geq r>R^\epsilon$ and the value of $u\in (0,1]$ plays no role. {We define $D^{\complement}$ to be $\mathcal{D}\setminus D$.}
Given two functions $f,g\colon D \to \R$ defined on a domain $D$, we write 
\begin{align*}
f(u,r,R) & \lesssim g(u,r,R) \text{ on }D, 
\end{align*}
if there exist $C, R_0 >1$ such that $f(u,r,R) \leq C g(u,r,R)$ for all $(u,r,R)\in D$ with $R\ge R_0$. We write 
\begin{align*}
f(u,r,R) & \asymp g(u,r,R) \text{ on } D,
\end{align*}
if both $f(u,r,R) \lesssim g(u,r,R)$ and $g(u,r,R) \lesssim f(u,r,R)$ on $D$. If no domain $D$ is mentioned the equivalence is supposed to hold on all of $\mathcal D$.
\pagebreak[3]

\subsection{Main results}



Recall that we assume that $\delta>2$ and $\gamma<1/2$, so that there exists a quantitatively subcritical phase, namely \smash{$\widehat\lambda_c>0$}. In this context, as mentioned above, for any \smash{$\lambda<\widehat\lambda_c$} we have the crossing probability decay 
\[
\theta_r=r^{d \zeta (1+o(1))} \text{ as } r\to \infty,
\]
with $\zeta$ defined as in Section~\ref{sec:background}. 
In \cite{jacob2025}, this result is obtained as a consequence of a similar result for the  
one-edge crossing probability $\pi_r$. By a  slight refinement of their argument, we obtain the following result, which we prove in Section~\ref{sec:Proof_of_Lemma_1.1}.

\begin{proposition}\label{prop:annulus_r_2r}
    In the quantitatively subcritical phase $0<\lambda<\widehat \lambda_c$, we have
    \[
    \theta_r \asymp \pi_r.
    \]
\end{proposition}

\begin{remark}
    Instead of $0<\lambda<\widehat \lambda_c$ the proof only uses the (potentially weaker) assumption that $\theta_r(\lambda)\to0$ as $r\nearrow\infty$. Under this assumption the precise bounds for $\pi_r$ given in Section~\ref{sec:one_two_crossing_bounds}
    give the up-to-constants asymptotics of $\theta_r$. The sharpness of the phase transition arises from the fact that for $\lambda<\widehat \lambda_c$ the crossing probability $\theta_r$ decays polynomially with a rate not depending on $\lambda$, whereas for $\lambda>\widehat \lambda_c$ 
    it remains of order one.  
\end{remark}

Our main interest is to consider thicker annuli, so we generalize the definition of $\theta_r$ to
\begin{align*}
	\theta_{r,R}:= \P(B(0,r)\longleftrightarrow \null B(0,R)^\complement),
\end{align*}
and that of $\pi_r$ to
\begin{align*}
	\pi_{r,R}:=\P(B(0,r)^{_{\simone}}  B(0,R)^\complement),
\end{align*}
for any $R\geq r>0$. Note that $\pi_{{r}}=\pi_{r,2r}$. 
We also define the two-edge  crossing probability, that is the probability of having a crossing path with at most two edges, by 
\begin{align*}
	\pi^{(2)}_{r,R}:= \P(B(0,r)^{_{\simtwo}} \null B(0,R)^\complement).
\end{align*}
While, for all $\lambda>0$, we have
\begin{equation}\label{eq:annulus_r_2r_crossing_with_one_or_two_edges}
    \pi_{r,2r}^{(2)}\asymp \pi_{r,2r},
\end{equation}
the analogous statement for thicker annuli does not necessarily hold if the inner radius $r$ is of smaller order than the  outer radius $R$. Bounding $\theta_{r,R}$ therefore involves the study of the two-edge crossing events. Further, we also state bounds for the probability of having a path from one fixed point to the outside of a ball of radius $R$.
Given a point $\x = (x,u) \in \R^d \times (0,1]$, we define the rooted graph $(\network, \x)$ by first constructing $\network$, and then, if $\x \notin \network$, adding~$\x$ to the vertex set and sampling its adjacent edges as usual. Given a set $A \subseteq \R^d$, we write $\x \longleftrightarrow A$ if there exists a finite-length path in $(\network, \x)$ from the root~$\x$ to a vertex in $A \times (0,1]$.  
We write
\[ \x^{_{\, \simk}} A,\]
if there is such a path with at most $k$ edges. 

For $u\in (0,1]$ we define the \emph{escape probabilities} from a ball with radius $R>0$ as
\begin{align*}
	\theta_{(u), R} & 
    :=\P((0,u)\longleftrightarrow  B(0,R)^\complement),\\
	\pi_{(u), R} & := \P((0,u)^{_{\simone }} B(0,R)^\complement),\\
	\pi^{(2)}_{(u), R} & := \P((0,u)^{_{\simtwo}}  B(0,R)^\complement).
\end{align*}
The idea behind these definitions is twofold: On the one hand, by averaging $\theta_{(u), R}$ over a uniform random mark~$u$
we get the escape probabilities from a ball centred in a typical point. On the other hand, by evaluating $\theta_{(u), R}$ for small deterministic values of~$u$ we check the effect of a strong vertex planted at the origin on the connectivity at scale $R$. This will be the essential tool to get lower bounds on the subcritical one-arm exponents in Theorem~\ref{thm:subcritical_one_arm_exponent}.%
\smallskip%

The following result, which is our main theorem, 
shows that in the quantitatively subcritical phase \emph{two} edges 
cross wide annuli or escape from balls with a strong centre with the same asymptotic probability as arbitrary paths. Recall that $\delta>2, \gamma<\frac12$ are necessary and sufficient conditions for the existence of a quantitatively subcritical phase.

\begin{theorem}\label{thm:two-edges}
    For the geometric inhomogeneous random graph  with $0\leq\gamma<\frac{1}{2}$ and either a hard profile and $\gamma\not=0$ or a soft profile with $\delta>2$ and with intensity \smash{$0< \lambda < \widehat{\lambda}_c$}, we have for all $\epsilon>0$, 
  \begin{align}
      \theta_{r, R} & \asymp \pi_{r, R}^{(2)}
      & \text{ on } \{  r \geq  R^\epsilon\},
      \label{eq:asymp_theta_r_R} \\[2mm]
      \theta_{(u), R} & \asymp 
                 \pi_{(u), R}^{(2)} 
                 & \text{ on } \{u \leq R^{-\epsilon}\}.
     \label{eq:asymp_theta_u_R}
  \end{align}  
\end{theorem}

\begin{remark}In the quantitatively subcritical phase, two-edge paths asymptotically yield the same crossing probabilities as paths with an arbitrary number of edges, irrespective of the relative growth rates of the inner and outer radii or the weight of the central vertex. In this phase the crossing probabilities $\theta_{r,R}$ decay polynomially with an order not depending on \smash{$0<\lambda<\widehat \lambda_c$}, whereas in the supercritical phase $\lambda>\lambda_c$ the $\theta_{r,R}$ are all of order one.  
\end{remark}
\begin{remark}Recall that Proposition~\ref{prop:annulus_r_2r} states that in the quantitatively subcritical phase we have $ \theta_{r,2r}\asymp \pi_{r,2r}$, which already gives our result for that particular choice of $R$, since a trivial inclusion of events implies $\pi_{r,R}\le \pi_{_{r,R}}^{_{(2)}}\le \theta_{r,R}$,
thus giving $ \theta_{r,2r}\asymp \pi^{_{(2)}}_{^{r,2r}}$. Conversely, we can recover Proposition~\ref{prop:annulus_r_2r} from our result by using~\eqref{eq:annulus_r_2r_crossing_with_one_or_two_edges}.\end{remark}

We find explicit up-to-constant bounds for the one and two-edge crossing probabilities, which combined with Theorem~\ref{thm:two-edges} gives a full picture showing when either 
one or two edges suffice to cross wide annuli with the same asymptotic probability as arbitrary paths. 

\begin{theorem}[Annulus crossing probabilities]\label{thm:main_R_alpha_R}
    For 
   $0\leq\gamma<\frac{1}{2}$
    and
    $0< \lambda < \widehat{\lambda}_c,$ we have, for 
    all $0< \epsilon < 1$, that:
    \smallskip   
    \begin{enumerate}
        \item If $2< \delta< 1/\gamma -1$, then 
        \begin{align*}
            \theta_{r, R} &  \asymp \pi_{r,R} \asymp r^dR^{d(1-\delta)} \qquad\text{ on } \{  
    r \geq  R^\epsilon\}.
        \end{align*}
        \item If $1/\gamma - 1< \delta < 1/\gamma$, then 
        \begin{align*}
            \theta_{r, R}& \asymp  \begin{cases}
            \pi_{r, R} 
            \asymp r^dR^{d(1-\delta)} & \text{ on } \{
            r \geq  R^{1+\delta-\frac 1 \gamma} \},\\                 \pi_{r, R}^{(2)} 
            \asymp r^{\frac{d}\delta}R^{d(1-\frac1\gamma)(1-\frac1\delta)} & \text{ on } 
                    \{ R^\epsilon \leq r \leq R^{1+\delta-\frac 1 \gamma}\}.\\
                \end{cases}
        \end{align*}
        \item If $\delta >1/\gamma$ or if $\delta = \infty$, then 
        \begin{align*}
            \theta_{r, R} & \asymp  \begin{cases}
            \pi_{r, R} 
            \asymp
                    r^d R^{d(1-\frac1\gamma)} \log R & \text{ on } \{ 
                    r \geq  R(\log R)^{\frac{-1}{d(1-\gamma)}} \},\\
                     \pi_{r, R}^{(2)} 
            \asymp  r^{d\gamma} R^{d(1-\frac1\gamma)(1-\gamma)} & \text{ on } \{ R^\epsilon\leq r \leq R (\log R)^{\frac{-1}{d(1-\gamma)}}\}.
                \end{cases}
        \end{align*}
    \end{enumerate}
\end{theorem}


\begin{remark}
The probability of crossing the annulus by an arbitrary path is of the same order as crossing it with one edge as soon as the annulus is not too thick for this strategy. If $\delta<{1}/{\gamma}$ the crossover of strategies occurs
    at \smash{$r\asymp R^{1+\delta-{1}/{\gamma}}$}. Hence there is no crossover in regime $(1)$ when 
    the profile has the heaviest tail and one edge always crosses with the same order of probability as an arbitrary path. The intermediate regime $(2)$ however splits into two parts depending on the 
    thickness of the annulus. If it is too thick two edges cross with the same order as an arbitrary path, but one does not. In the light tail regime $(3)$, 
    which includes the hard profile, the crossover occurs for much thinner annuli, at  \smash{$r\asymp R (\log R)^{{-1}/{d(1-\gamma)}}$}. Then the order of the annulus crossing probability does not depend on the profile tail index~$\delta$. 
    We always assume that the inner radius grows at least 
polynomially with the outer radius.%
\end{remark}%

\begin{remark}\label{rem:boundary}
 For the asymptotic crossing probabilities in the boundary cases $\delta\gamma=1-\gamma$ and $\delta\gamma=1$ see Section~\ref{sec:one_two_crossing_bounds}, more precisely Corollary~\ref{cor:2e-probabilities-boundarycases}.
\end{remark}

\begin{theorem}[Escape probabilities from balls with strong centre]\label{thm:main_u_R}
    For 
    $0\leq\gamma<\frac{1}{2}$
    and
    $0< \lambda < \widehat{\lambda}_c,$ we have, for all $\epsilon>0$, that \smallskip
    \begin{enumerate}
    \item[(1)] If $2 < \delta < 1/\gamma -1$, then  \begin{align*}   \theta_{(u), R} \asymp \pi_{(u), R}\asymp \begin{cases}          u^{-\gamma\delta}R^{d(1-\delta)} & \text{ on } \{R^{d(\frac1{\gamma\delta}-\frac 1 \gamma)}\leq u \leq R^{-\epsilon}\},\vspace{0.5mm}\\                 1 & \text{ on }                    \{ u \leq R^{d(\frac1{\gamma\delta}-\frac 1 \gamma)} \}.\\                \end{cases}
        \end{align*}\smallskip
     \item[(2)] If $1/\gamma - 1< \delta < 1/ \gamma$, then
     \begin{align*}
     \theta_{(u), R} \asymp 
       \begin{cases}
       \pi^{(2)}_{(u), R}\asymp
       u^{-\gamma}R^{d(1-\frac1\gamma)(1-\frac1\delta)} \phantom{IIJ}
            \text{ on } \{ R^{-\frac{d(1+\delta-\frac 1 \gamma)}{\gamma\delta}}\leq u \leq R^{-\epsilon}\}, \vspace{1mm}\\ 
          \pi_{(u), R}\asymp
          \begin{cases}
       u^{-\gamma\delta}R^{d(1-\delta)} \phantom{XX} & \text{ on } \{R^{d(\frac 1{\gamma\delta}-\frac1\gamma)}\leq u \leq
         \smash{R^{-\frac{d(1+\delta-\frac 1 \gamma)}{\gamma\delta}}}\}, \vspace{0.5mm}\\
        1 & \text{ on } 
               \{  u \leq R^{d(\frac 1{\gamma\delta}-\frac1\gamma)} \}.\\
              \end{cases}
             \end{cases}
      \end{align*}\smallskip
       \item[(3)] If $\delta> 1/\gamma$ or if $\delta = \infty$ (hard case), then 
       \begin{align*}
            \theta_{(u), R} \asymp 
          \begin{cases}
            \pi^{(2)}_{(u), R}\asymp
            u^{-\gamma}R^{d(1-\frac1\gamma)(1-\gamma)}
             \phantom{XJ}\text{ on } \{R^{-d}\leq u \leq R^{-\epsilon}\},\vspace{1mm}\\
           \pi_{(u), R}\asymp
           \begin{cases}
          u^{-1}R^{d(1-\frac1\gamma)} \phantom{XX} & \text{ on } \{R^{d(1-\frac1\gamma)}\leq u \leq R^{-d}\},\vspace{0.5mm}\\
           1 & \text{ on } 
                  \{ u\leq  R^{d(1-\frac1\gamma)} \}.\\
              \end{cases}
              \end{cases}
       \end{align*}
 \end{enumerate}
   \end{theorem}

\begin{remark}
 As in Remark~\ref{rem:boundary} the boundary cases $\delta\gamma=1-\gamma$ and $\delta\gamma=1$ are detailed in Corollary~\ref{cor:2e-probabilities-boundarycases} 
 of Section~\ref{sec:one_two_crossing_bounds}.
\end{remark}

Although in Theorem~\ref{thm:main_u_R} we plant unusually strong vertices at the origin, the result enables us to give lower bounds on the probability that the connected component of a typical point has a radius larger than  $R$ about its centre. To make sense of the notion of a typical point we use the Palm theory of marked Poisson processes, which allows us to pick a typical vertex from the graph and shift it to the origin.
\smallskip

\begin{theorem}[Subcritical one-arm exponents]\label{thm:subcritical_one_arm_exponent}
    For $0\leq\gamma<\frac12$ and 
    $0< \lambda < \widehat{\lambda}_c,$ as $R\nearrow\infty$,
    $$\frac1{d\log R} \log \P^*((0, U) \longleftrightarrow B(0,R)^\complement)
\longrightarrow
\begin{cases}
{1-\delta} & \text{if $2 < \delta < 1/\gamma -1$,}\\
(1-\frac1\gamma)(1-\frac1\delta)
            & \text{if $1/\gamma - 1\leq \delta < 1/ \gamma$,}\\
(1-\frac1\gamma)(1-\gamma)
& \text{if $\delta \ge 1/ \gamma$ or $\delta = \infty$,}
\end{cases}
$$
where $\mathbb P^*$ is the Palm measure at the origin associated with the marked Poisson point process on which our graph is defined, and $(0,U)$ is the vertex located at the origin under $\mathbb P^*$. 
\end{theorem}

\begin{remark}
Equivalent ways to define $\mathbb P^*$ are (i) sample a standard uniform $U$ and add the point $(0,U)$ to the vertex set, (ii) pick a point at random from all points located in $B(0,n)$, shift it to the origin and take a limit  of the probability as $n\to\infty$,  (iii) condition the Poisson process on having a point located at the origin. To do so, condition the Poisson process on having a point 
with location in $B(0,\eps)$, and let $\eps$ go to 0. For details see~\cite{LP}.
\end{remark}

\begin{remark}
    The event that the connected component of a typical point has a radius larger than  $R$ about its centre has polynomially decaying probability and the rate does not depend on $\lambda<\widehat{\lambda}_c$. 
  It is easy to see that the same result also holds for the diameter of a typical component. The result also holds under the (potentially weaker) assumption $\theta_R(\lambda)\to0$ and therefore  
    under this assumption the one-arm exponents extend to \smash{$\lambda=\widehat{\lambda}_c$}. 
\end{remark}

\begin{remark}
    In fact, we prove a sharper lower bound since
    \[
    \P^*((0, U) \longleftrightarrow B(0,R)^\complement) \ge \P^*((0, U) \simtwo B(0,R)^\complement)\asymp \pi_{(1),R}^{(2)},
    \]
    with asymptotics provided in Table~\ref{tab:gandh}. We also conjecture that Theorem~\ref{thm:two-edges} holds for $\epsilon=0$, in which case we would get
    $
    \P^*((0, U) \longleftrightarrow B(0,R)^\complement) \asymp \pi_{^{(1),R}}^{_{(2)}}.
    $
\end{remark}

\begin{remark}
Jahnel et al.~\cite{JLO} identify the one-arm exponents for the soft Boolean model for sufficiently small $\lambda>0$. In a comparable parametrisation they obtain
$$\frac1{d\log R} \log \P^*((0, U) \longleftrightarrow B(0,R)^\complement)
\longrightarrow
\begin{cases}
{1-\delta} & \text{if $1 < \delta \leq  1/\gamma -1$,}\\
(1-\frac1\gamma)(1-\frac1\delta)
            & \text{if $1/(\gamma \wedge (1-\gamma)) - 1< \delta$.}\\
\end{cases}
$$
Interestingly,  the exponents agree for $2<\delta\leq1/\gamma$ but for $\delta>1/\gamma$ they differ.
Their method of proof also differs considerably from ours.
\end{remark}

\pagebreak[3]

\textbf{Outline.} We start the proofs in Section~\ref{sec:prelim} with preliminary considerations on the dependence of the two-edge crossing probabilities on the radii~$r,R$ and intensity~$\lambda$ and the proof of Proposition~\ref{prop:annulus_r_2r}. In Section~\ref{sec:one_two_crossing_bounds} we prove tight up-to-constants bounds for the one-and two-edge crossing and escape probabilities, see Proposition~\ref{prop:asymp_escape_crossing_probabilities}. In particular, in Corollary~\ref{cor:2e-probabilities-boundarycases}, we state the bounds 
in the boundary cases thus complementing the statement of Theorems~\ref{thm:main_R_alpha_R} and~\ref{thm:main_u_R}.
We also provide bounds on the expected number of three-edge connections in Lemma~\ref{lem:integral_edge_crossing}.
In Section~\ref{sec:bootstrap} we present the novel method behind the proof of Theorem~\ref{thm:two-edges}. We begin with presenting the geometric argument which is at the heart of our proof in Section~\ref{subsec:geometric_argument}. The implications in
Proposition~\ref{prop:bootstrapping} allow us, in Section~\ref{subsec:bootstrap_argument}, to iteratively increase the domain on which the required bounds hold using a bootstrapping argument stated in Proposition~\ref{prop:induction_on_domain}.
We conclude the proofs of our four main theorems in Section~\ref{sec:thms} by combining the bootstrapping argument with the bounds of Section~\ref{sec:one_two_crossing_bounds}.

\section{Preliminary considerations}
\label{sec:prelim}

\subsection{Continuity and monotonicity of crossing probabilities}

\begin{lemma}\label{lem:monotone}
For every $R>0$ the map $(0,R) \to (0,\infty), r \mapsto \theta_{r,R}$ is strictly increasing and the map
$(0,1] \to (0,\infty), u \mapsto \theta_{(u),R}$ is strictly decreasing, the maps $\{0<r<R\} \to (0,\infty), (r,R)\mapsto \theta_{r,R}$ and $(u,R)\mapsto \theta_{(u),R}$ are continuous.
\end{lemma}

\begin{proof}
For $r<r'<R$ the event $\{ B(0,r) \leftrightarrow B(0,R)^\complement\}$ implies 
$\{ B(0,r') \leftrightarrow B(0,R)^\complement\}$ giving monotonicity. It is strict as with positive probability the latter holds while  $B(0,r)\cap\poisson =\emptyset$. Similarly, under a straightforward coupling
for $u<u'$ the event $\{ (0,u') \longleftrightarrow B(0,R)^\complement\}$ implies 
$\{ (0,u) \longleftrightarrow B(0,R)^\complement\}$. The monotonicity is strict, as with positive probability there is exactly one edge adjacent to the origin for $u$ but none for $u'$.
Observe that for $r'\le r\le r''\le R''\le R\le R'$ the event $\{ B(0,r'') \leftrightarrow B(0,R'')^\complement\} \setminus \{ B(0,r') \leftrightarrow B(0,R')^\complement\}$ implies that there is a vertex in
$B(0,r'')\setminus B(0,r')$ or in $B(0,R')\setminus B(0,R'')$. The probability of this event goes to zero as $r',r''\to r$ and $R', R'' \to R$, and hence $0\leq \theta_{r'',\,R''}-\theta_{r',\,R'}\to 0$. This implies continuity as $\theta_{r',R'}\le \theta_{r,R}\le \theta_{r'',R''}.$ Continuity of $\theta_{(u),R}$ is proved following similar steps, using a straightforward coupling for the graphs with different values of~$u$.
\end{proof}

\begin{remark}\label{rem:continuity_monotonicity}
All statements of the lemma also hold for $\pi_{r,R}$ and $\pi^{_{(2)}}_{^{r,R}}$ in place of~$\theta_{r,R}$, and $\pi_{(u),R}$ and \smash{$\pi^{_{(2)}}_{^{(u),R}}$} in place of~$\theta_{(u),R}$.
\end{remark}

\subsection{Scaling finite-crossing events}
Many of the events we consider involve only a finite number of vertices and edges in the graph. 
In geometric inhomogeneous random graphs the asymptotic order of probability of such events is typically not changed if we modify the intensity $\lambda$ of the Poisson point process, or the scale at which the event is considered. 

\begin{definition}
    An event $E$ is called a \emph{finite-crossing event} for a graph $\mathscr G$ with vertex set $V\subset\R^d \times [0,1]$ if there exists $k\ge 1$, and a Borel subset $A \subset (\R^d \times (0,1])^k$ 
    such that
    \[
    E=\big\{\exists (\x_1,\ldots,\x_k)\in V^k\cap A, \forall 1\leq i\leq k-1,\, \x_i=\x_{i+1}\text{ or }\,\x_i\sim \x_{i+1}\big\}
    \]
    We also use this definition for rooted graphs
    $(\mathscr G, \x_0)$ adding the requirement $\x_0\sim \x_1$
    to the definition of $E$. Given $k$ and $A$ we say $E$ is the finite-crossing event associated with~$(k,A)$.%
    \pagebreak[3]%
    
    The set
    $A$ is called \emph{downward-closed} if, for all  $((x_1,u_1),\ldots,(x_k,u_k)) \in A$ and $u_1' \leq u_1, u_2' \leq u_2, \ldots, u_k' \leq u_k$ it follows that $((x_1,u_1'),\ldots,(x_k,u_k')) \in A$. The event $E$ is called downward-closed if $A$ is.
\end{definition}\pagebreak[3]

Note that the definition of finite-crossing events is broad enough to include events of the form $A^{_{\,\simk}} B$ and $(x,u)^{_{\, \simk}} A$, the latter using the notion for the graph with root $(x,u)$, and that both events are downward-closed. However, the events $A\longleftrightarrow B$ and $(x,u)\longleftrightarrow A$ are not finite-crossing events, since these consider paths of arbitrary length. The important feature of this definition is that a finite-crossing event $E$ associated with 
$(k,A)$ depends only on $k$ vertices and at most $k-1$ edges observed between these vertices, thus allowing us to control the impact of rescaling $\R^d$.\medskip

Let $A\subset (\R^d \times (0,1])^k$  be a Borel set as before, and $c>0$. We define $cA$ as
\[cA:=\{(cx,u) \colon (x,u)\in A\},\]
that is, the set defined by rescaling by $c$ the location of all the points in $A$. If $E$ is the finite-crossing event associated with $(k,A)$ and $c>0$, we write $cE$ to denote the finite-crossing event associated with 
$(k,cA)$. In the special case where $E$ is defined on $(\network,(x,u))$, $cE$ is defined on $(\network,(cx,u))$, that is, we also rescale the location of the root. 

\begin{lemma}
    Fix $p\in(0,1)$ and let $E$ be a finite-crossing event associated with some $(k,A)$. Then, for any decay exponent $1<\delta\le\infty$, we have
    \begin{eqnarray}
        \label{scaling_plambda}    p^k \P^{(\lambda)}(E)&\le \P^{(p\lambda)}(E) &\le \P^{(\lambda)}(E), \\
        \label{scaling_pE} (\tfrac{\mathtt c}{\mathtt C}p^{d\delta})^{k-1}\P^{(p^{-d}\lambda)}(E)&\le \P^{(\lambda)}(p^{-1}E)&\le \P^{(p^{-d}\lambda)}(E),
    \end{eqnarray}
    where the left-hand side in~\eqref{scaling_pE} uses $\mathtt c, \mathtt C$ from \eqref{eq:conditions_rho} for the soft profile and is set equal to zero for the hard profile. If $E$ is downward closed and $\gamma>0$, we also have
    \begin{equation}\label{scaling_downwardclosed}
    \P^{(p^{d(1/\gamma-1)}\lambda)}(E)\le \P^{(\lambda)}(p^{-1}E).
    \end{equation}
\end{lemma}

\begin{proof}
    We first give the complete proof for the case where $E$ is an event on the unrooted graph $\network$, and then describe the necessary modifications to extend it to $(\network,(x_0,u_0))$. As we will consider several values of~$\lambda$, we write $\mathscr G_\lambda$ for the graph obtained for a fixed value of~$\lambda$. The inequalities in \eqref{scaling_plambda} follow from the observation that the graph $\mathscr G_{p\lambda}$ can be obtained by first constructing the graph $\mathscr G_{\lambda}$, and then performing  independent Bernoulli percolation with retention parameter $p$ on its vertices. In this construction, if $E$ holds on $\network_{p\lambda}$, then it must also hold on $\network_{\lambda}$. Conversely, any sequence of vertices satisfying $E$ in $\mathscr G_{\lambda}$ is also present in $\mathscr G_{p\lambda}$ with probability at least $p^k$.
    \smallskip

    For the proof of \eqref{scaling_pE} 
    we define a coupling of $\network_{\lambda}$ and $\network_{p^{-d}\lambda}$. Fix a realization of $\network_{p^{-d}\lambda}$ and define an intermediate graph \smash{$\widetilde{\network}$} by moving each vertex \smash{$(x,u)\in\network_{p^{-d}\lambda}$} to a new location $(p^{-1}x,u)$, while preserving all the edges. Construct $\network_{\lambda}$ by performing  independent edge percolation on $\widetilde{\network}$, with each edge $\{(x,u),(y,v)\}$ being retained with probability 
    \[
    \frac{\rho\left(p^{-d}|y-x|^d (uv)^\gamma\right)}{\rho\left(|y-x|^d (uv)^\gamma\right)}\geq \frac{\mathtt c}{\mathtt C}p^{d\delta},
    \]
    where the inequality follows from \eqref{eq:conditions_rho} in the case of a soft profile. It can be readily checked that this alternative construction of $\network_{\lambda}$ gives the same law as the original one. From this construction, we deduce that $E$ occurs in $\network_{p^{-d}\lambda}$ if and only if $p^{-1}E$ occurs in $\widetilde{\network}$, and the result then follows from the same argument as in the proof of \eqref{scaling_plambda}. \smallskip

    The proof of \eqref{scaling_downwardclosed} follows, again, from a coupling of \smash{$\network_{p^{d(1/\gamma-1)}\lambda}$} and $\network_\lambda$. The construction is a bit more involved than the one used for \eqref{scaling_pE}.
    \begin{enumerate}
        \item Fix a realization of \smash{$\network_{p^{d(1/\gamma-1)}\lambda}$} and define an intermediate graph $\widetilde{\network}$ by moving each vertex $(x,u)\in\network_{p^{d(1/\gamma-1)}\lambda}$ to $(p^{-1}x,p^{d/\gamma}u)$, while preserving all the edges.
        
        \item For every pair of vertices $\{(p^{-1}x,p^{d/\gamma}u),(p^{-1}y,p^{d/\gamma}v)\}$ in $\widetilde{\network}$ not initially adjacent, we resample the edge with probability
         \[
    \frac{\rho\left(p^{d}|y-x|^d (uv)^\gamma\right)-\rho\left(|y-x|^d (uv)^\gamma\right)}{1-\rho\left(|y-x|^d (uv)^\gamma\right)}.
    \]
    This probability is nonnegative by our assumptions on $\rho$.\pagebreak[3]\smallskip
    
    \item Finally, we construct $\network_{\lambda}$ by adding to the vertex set of $\widetilde{\network}$ the points of a Poisson point process $\poisson'$ on $\R^d\times[p^{d/\gamma},1)$ with intensity $\lambda$, and sampling edges between vertices in $\poisson'$ and those of $\widetilde{\network}$ as before.
        \end{enumerate}
    It can be verified that this  construction yields a sample of $\network_{\lambda}$. Suppose that $E$ is a downward closed finite-crossing event, and that $(x_1,u_1),\ldots,(x_k,u_k)$ is a sequence of vertices in \smash{$\network_{p^{d(1/\gamma-1)}\lambda}$} satisfying $E$. Then the sequence $(p^{-1}x_1,p^{d/\gamma}u_1),\ldots,(p^{-1}x_k,p^{d/\gamma}u_k)$ obtained after Step~(1)
    satisfies $p^{-1}E$ in $\network_{\lambda}$, so from this coupling we conclude \eqref{scaling_downwardclosed}.\smallskip
    
    The proof of the lemma for events $E$ defined on $(\network, (x_0,u_0))$ is analogous; however, special treatment is needed for the root. In the proof of~\eqref{scaling_pE}, the root is always kept during site percolation. In the proof of~\eqref{scaling_plambda}, we keep its adjacent edges. In the proof of \eqref{scaling_downwardclosed}, the root is relocated to $(p^{-1}x_0, u_0)$ instead of $(p^{-1}x_0, p^{d/\gamma}u_0)$, and no edges are added adjacent to it in Step $(2)$. With these considerations, the proof carries over analogously.
\end{proof}
The next lemma states that linear rescaling of the inner or outer radii does not affect the order of crossing and escape probabilities.
\begin{lemma}\label{lem:renormalization_constants}
Let $0\leq\gamma<1$ and $1<\delta \leq\infty$ and assume $\gamma = 0$ and $\delta = \infty$ do not occur simultaneously. For every $a>0$ there exist constants $0<c(a)<C(a)$ (depending on $\gamma$ and $\delta$) such that
for all $0<u<1$  and $R>0$ we have
\begin{align}
    \label{scaling_piu1} c(a) \pi_{(u),R} \leq & \pi_{(u),aR} \leq C(a) \pi_{(u),R}, \\
    \label{scaling_piu2}c(a) \pi^{(2)}_{(u),R} \leq & \pi^{(2)}_{(u),aR} \leq C(a) \pi^{(2)}_{(u),R}.
\end{align}
 Moreover, for every $\varepsilon>0$ and $a,b>0$, there exist constants 
 $0<c(a,b,\varepsilon)<C(a,b,\varepsilon)$ (again, depending on $\gamma$ and $\delta$) such that, for $r,R>0$ with 
 $0<\frac{r}{R}\leq (1-\varepsilon)(\frac{b}{a}\wedge1)$, we have
 \begin{align}
     \label{scaling_pi1}c(a,b,\varepsilon) \pi_{r,R} \leq & \pi_{ar, bR} \leq C(a,b,\varepsilon) \pi_{r,R}, \\
     \label{scaling_pi2}c(a,b,\varepsilon) \pi^{(2)}_{r,R} \leq & \pi^{(2)}_{ar, bR} \leq C(a,b,\varepsilon) \pi^{(2)}_{r,R}.
 \end{align}

\end{lemma}


\begin{proof}
Combining~\eqref{scaling_pE} and~\eqref{scaling_downwardclosed}, we obtain that for any downward closed finite-crossing event $E$ associated to some $(k,A)$, and any $0<p<1$,
\[
    \max\big\{(\tfrac{\mathtt c}{\mathtt C}p^{d \delta})^{k-1}\P^{(p^{-d}\lambda)}(E),\P^{(p^{d(1/\gamma-1)}\lambda)}(E)\big\}\le \P^{(\lambda)}(p^{-1}E)\le \P^{(p^{-d}\lambda)}(E).  
\]
Combining these inequalities with~\eqref{scaling_plambda} we obtain that
\[
     \max\big\{(\tfrac{\mathtt c}{\mathtt C}p^{d \delta})^{k-1},p^{d(1/\gamma-1)k}\big\}\P^{(\lambda)}(E)\le \P^{(\lambda)}(p^{-1}E)\le p^{-kd}\P^{(\lambda)}(E).  
\]
Since either $\gamma>0$ or $\delta<\infty$, we have found explicit constants $0<c\leq C$ depending on $\gamma$, $\delta$, $p$ and $k$, such that
\begin{equation}
    \label{eq:scaling}
     c\P(E)\le \P(p^{-1}E)\le C\P(E).  
\end{equation}
Notice that these inequalities hold for $0<p<1$, however we can extend them for general $p>0$. Indeed, the case $p=1$ is trivial, and for the case $p>1$ it is enough to observe that $E=p(p^{-1}E)$ so we can write
\[c\P(p^{-1}E)\le \P(E)\le C\P(p^{-1}E),\]
with $c$ and $C$ computed with $p^{-1}$, and thus \eqref{eq:scaling} follows in this case by rearranging the inequalities. The inequalities~\eqref{scaling_piu1} and~\eqref{scaling_piu2} follow directly from~\eqref{eq:scaling} by setting $p = a$ and $E = \{(0,u)^{_{\, \simk}} B(0,R)^{\complement}\}$, with $k = 1,2$, respectively.
\smallskip

The proofs of~\eqref{scaling_pi1} and~\eqref{scaling_pi2} also rely on~\eqref{eq:scaling}, 
however, the argument is more involved and must be handled differently depending on the relationship between the parameters $a$ and $b$. We provide only the proof of~\eqref{scaling_pi1}, as the proof of~\eqref{scaling_pi2} is almost identical.
\smallskip  

    \emph{Case $a=b$:}
The result follows directly from~\eqref{eq:scaling} by taking $E
=\{B(0,r)^{_{\simone}} B(0,R)^{\complement}\}$ and $p=1/a$ so that
$\P(E)=\pi_{r,R}$ and $\P(p^{-1}E)=\pi_{ar,aR}$. 
\medskip

    \emph{Case $a<b$:} From the previous case, we may assume $b = 1$, so that $a < 1$. Hence, $\pi_{ar, R} \leq \pi_{r, R}$, which gives the inequality on the right-hand side of~\eqref{scaling_pi1}. To bound $\pi_{ar, R}$ from below we first note that there is an integer $n$ depending only on~$a$, $\varepsilon$, and the dimension, such that $B(0,r)$ can be covered by $n(a)$ balls of radius $\varepsilon a r$ centred in $B(0,r)$. If $B(0,r)^{_{\simone}} B(0,R)^{\complement}$
     occurs, then there must be two adjacent vertices $\x_1\sim\x_2$ such that  $x_1\in B(0,r)$ and \smash{$x_2\in B(0,R)^\complement$}. Call $B(y,\varepsilon a r)$ the covering ball containing $x_1$ and notice that $|x_2-y|\geq R-r\geq\varepsilon R $, where the last inequality follows from our assumption $r\leq R(1-\varepsilon)$. It then follows that $x_2\in B(x,\varepsilon R)^{\complement}$ and hence the vertices $\x_1,\x_2$ satisfy $B(x,\varepsilon aR)^{_{\simone}} B(x,\varepsilon R)^{\complement}$. Taking a union bound over all the $n$ balls $B(y,\varepsilon ar)$, and using~\eqref{eq:scaling} with $E=\{B(0,ar)^{_{\simone}} B(0,R)^{\complement}\}$ and $p=\varepsilon$ we obtain
     \[\pi_{r,R} \leq n\pi_{\varepsilon a r,\varepsilon R} \leq nC \pi_{ar,R}\]

    \emph{Case $a>b$:} Call $r'=ar$ and $R'=bR$ and observe that from our hypothesis on $r$ and $R$ these must satisfy $r'\leq (1-\varepsilon)R$. Since $\frac{1}{a}<\frac{1}{b}$, we can use~\eqref{scaling_pi1} with $r',R'$ to obtain 
    \[c(\tfrac{1}{a},\tfrac{1}{b},\varepsilon) \pi_{r',R'} \leq  \pi_{r, R} \leq C(\tfrac{1}{a},\tfrac{1}{b},\varepsilon) \pi_{r',R'},\]
and the result then follows by rearranging these inequalities.
\end{proof}

\subsection{Proof of Proposition~\ref{prop:annulus_r_2r}}
\label{sec:Proof_of_Lemma_1.1}

    The proof of Proposition~\ref{prop:annulus_r_2r} is essentially hidden in the findings of \cite{jacob2025}. 
    In \cite[Definition 1.4]{jacob2025}, the authors introduce the 
    event 
    \[
    \mathcal{L}(r,c):=\big\{\exists\ \x\sim \y\colon |x|<r, |x-y|>cr \big\},
    \]
    and prove in~(15) that
    $
        \theta_r\le  \P(\mathcal{L}(2r,1/20))+ C_1 \theta_{r/10}^2,
    $
    for some constant $C_1$ depending only on the dimension, and in particular independent of~$r$. This 
    renormalization inequality is at the heart of the definition and study of the quantitatively subcritical phase $\lambda<\widehat \lambda_c$. Indeed, if $\P(\mathcal{L}(2r,1/20))$ tends to 0 as $r$ tends to infinity, a property which does not depend on the value of $\lambda>0$ by~\eqref{scaling_plambda}, and if $\theta_r$ is sufficiently small for some value of $r$, then $\theta_r$ will automatically decay to 0 with decay being controlled by $\P(\mathcal{L}(2r,1/20))$.
    More precisely, \cite{jacob2025} shows that in the case $\zeta<0$, for any given $\alpha>0$ and $c>0$, we have
    \[
    \P(\mathcal{L}(\alpha r,c))= r^{d\zeta(1+o(1))}.
    \]
    Then they use the renormalization inequality to show first that $\widehat \lambda_c>0$, and moreover that for any $\lambda<\widehat \lambda_c$ the term $\theta_r$ must also decay as $r^{d\zeta(1+o(1))}.$ 
    In particular, the term $C_1 \theta_{r/10}^2$ is negligible, and we obtain the upper bound
    \[
    \theta_r\lesssim  \P(\mathcal{L}(2r,1/20)).
    \]
    It is elementary that $\P(\mathcal{L}(r,c))\asymp \P(\mathcal{L}(r',c'))$ if $cr=c'r'$. Indeed, in the two events we look for edges of same length in balls of slightly different sizes and we can include the smaller ball in the larger ball and the larger in a fixed number of smaller balls. We can thus continue this upper bound with
    \[
    \theta_r\lesssim \P(\mathcal{L}(r/30,3))\le \pi_{r/30}.
    \]
    where the last inequality follows by inclusion of events. By~\eqref{eq:scaling}, we also have
    $\pi_{r/30}\asymp \pi_r$ and combining this with the elementary inequality $ \pi_r\le \theta_r$, we deduce $\theta_r\asymp \pi_r\asymp \P(\mathcal{L}(r,1))$.

    
\section{One and two-edge crossing probabilities}\label{sec:one_two_crossing_bounds}

In this section, we derive up-to constant bounds for the probabilities of one-edge and two-edge crossings. These results hold for arbitrary $\lambda >0$ and 
not just in the quantitatively subcritical phase. Indeed, by~\eqref{scaling_plambda} the results involving only finite crossing events do not depend on the value of $\lambda>0$. The main bounds are collectively stated in Section~\ref{sec:state_one_two_crossing_bounds}. Proofs for the one and two-edge probabilities are given in Section~\ref{subsec:3.4}, while bounds involving three-edge crossings appear in Section~\ref{sec:integral_edge_crossing}.

\subsection{Statement of the one and two-edge crossing probabilities}
\label{sec:state_one_two_crossing_bounds}

%
We first argue that if $\delta=2$  or $\gamma=\frac12$ there is no quantitatively subcritical phase. Recall from the introduction that in this case $\zeta=0$ and the results of~\cite{jacob2025} do not settle whether \smash{$\widehat \lambda_c$} is positive or 0. The following lemma resolves this question. 

\begin{lemma}\label{lem:quant}
    If $\delta=2$  or $\gamma=\frac12$, then the one-edge crossing probabilities satisfy
    $\pi_R\asymp 1$. 
\end{lemma}

From now on we assume that $\delta>2$ and $\gamma<\frac12$ such that a quantitatively subcritical phase exists. The following proposition summarizes the main results of this section. 
\begin{proposition}\label{prop:asymp_escape_crossing_probabilities}
Let $2<\delta\leq\infty$ and $0\leq\gamma<\frac12$, excluding the case $(\gamma,\delta)=(0,\infty)$, and adopt the convention $\frac{1}{\gamma}=\infty$ when $\gamma=0$. Then, for arbitrary $\lambda>0$,
    the one-edge escape and crossing probabilities satisfy 
    \begin{align} \label{asymp:1e-escape-probabilities}
        \pi_{(u), R} \asymp 
        \begin{cases}
            u^{-1}g(R) \wedge 1 & \mbox{ if } \delta\ge\frac1\gamma, \\
            u^{-\gamma\delta}g(R) \wedge 1 & \mbox{ if } \delta<\frac1\gamma,
        \end{cases}
    \end{align}
    and
    \begin{align} \label{asymp:1e-crossing-probabilites}
        \pi_{r, R} \asymp 
        \begin{cases}
            r^d (\log R) g(R)  & \mbox{ if } \delta\ge\frac1\gamma, \\
            r^d g(R)  & \mbox{ if } \delta<\frac1\gamma,
        \end{cases}
    \end{align}
    and the two-edge escape and crossing probabilities satisfy
    \begin{align} \label{asymp:2e-escape-probabilities}
        \pi_{(u), R}^{(2)} \asymp 
        \begin{cases}
        \pi_{(u),R} + u^{-\gamma}h(R)\wedge 1 & \mbox{ if } \delta\ge \frac1\gamma-1,\\ \pi_{(u),R} & \mbox{ if } 
        \delta< \frac1\gamma-1,
        \end{cases}
    \end{align}
    and
 \begin{align} \label{asymp:2e-crossing-probabilites}
        \pi^{(2)}_{r, R} \asymp 
        \begin{cases}
            \pi_{r,R} + 
            r^{d\gamma} h(R)   & \mbox{ if } \delta>\frac1\gamma,\vspace{0.5mm} \\
             \pi_{r,R} + r^{d\gamma} (\log r)^\gamma h(R)   & \mbox{ if } \delta=\frac1\gamma,\vspace{0.5mm} \\
            \pi_{r,R} + r^{\frac{d}{\delta}} h(R)  & \mbox{ if } \frac{1}{\gamma}-1\leq\delta<\frac1\gamma, \vspace{0.5mm}\\
             \pi_{r,R}    & \mbox{ if } \delta<{\frac1\gamma-1},
        \end{cases}
    \end{align}    
    where the {asymptotic order of} $g(R):=\pi_{(1),R}$ and $h(R):=\pi^{(2)}_{(1),R}$ are stated in Table~\ref{tab:gandh}.
\end{proposition}

{\renewcommand{\arraystretch}{2}%
\begin{table}[h]
\begin{NiceTabular}{| l | l || l | l  |  }[hvlines]
 Regime &  Order of $g(R)$ \phantom{XX} &  Regime &  Order of $h(R)$ \phantom{XX} \\
 $\infty \ge \delta > 1/\gamma$ & $R^{d(1 - 1/\gamma)}$ &
 $\infty \ge \delta > 1/\gamma$ & $R^{d(1-\gamma)(1 - 1/\gamma)}$ 
 \\
 $\delta = 1/\gamma$ & $R^{d(1-\delta)}\log R$ &
 $\delta = 1/\gamma$ & $R^{d(1 -\gamma)(1-\delta)}(\log R)^{1-\gamma} $
 \\
\Block{3-1}{$2< \delta < {1}/{\gamma}$} & \Block{3-1}{$R^{d(1-\delta)}$} &
${1}/{\gamma} - 1 < \delta < {1}/{\gamma}$ & $R^{d(1-1/\gamma)(1-1/\delta)}$ \\
  &  &
 $\delta={1}/{\gamma} - 1 $ & $R^{d(1-\delta)}\log R$ \\
& &
 $2<\delta< {1}/{\gamma} - 1 $ & $R^{d(1-\delta)}$ \\
\end{NiceTabular}
\caption{Overview of the values of $g(R):=\pi_{(1),R}$ and $h(R):=\pi^{(2)}_{(1),R}$. }
\label{tab:gandh}
\end{table}}

\begin{remark}
Note that the asymptotics $\pi_{^{(u), R}}^{_{(2)}} \asymp \pi_{(u),R} + u^{-\gamma}h(R)\wedge 1$ actually holds for all values of $\gamma$, $\delta$, however in the case $\delta<1/\gamma-1$ where we have $h(R)\asymp g(R)$, the second term $u^{-\gamma}h(R)$ is always dominated by the first term $u^{-\gamma \delta} g(R)$, and therefore can be removed. In that case, the one and two-edge crossing probabilities have the same asymptotics on the whole domain $\mathcal D$, i.e.
    $\pi_{^{r,R}}^{_{(2)}}\asymp \pi_{r,R}$ and $\pi_{^{(u),R}}^{_{(2)}}\asymp \pi_{(u),R}$.
\end{remark}\pagebreak[3]

For given $\gamma$ and $\delta,$ it is a  simple task to identify which is the dominant term in the asymptotics~\eqref{asymp:2e-escape-probabilities} and~\eqref{asymp:2e-crossing-probabilites}. In the boundary cases $\delta=1/\gamma$ and $\delta=1/\gamma-1$, gives the following corollary, complementing the results already stated in Theorem~\ref{thm:main_R_alpha_R} and~\ref{thm:main_u_R}. 

\begin{corollary}
    \label{cor:2e-probabilities-boundarycases} 
    When $\delta=1/\gamma$, the two-edge crossing probabilities satisfy
    \begin{align*}
        \pi_{(u),R}^{(2)} 
        &\asymp \begin{cases}
                    u^{-\gamma} h(R) , & \text{ on } \{R^{-d}{(\log R)^{\frac{\gamma}{1-\gamma}}}\le u\},\vspace{0.3mm}\\
                    u^{-1} g(R) & \text{ on } \{ R^{d(1-\frac 1 \gamma)} \log R\le u\le R^{-d}{(\log R)^{\frac{\gamma}{1-\gamma}}}\},\vspace{0.3mm}\\
                    1 & \text{ on } \{u\le R^{d(1-\frac 1 \gamma)}\log R\}.
                \end{cases}\\
        \pi_{r,R}^{(2)}
        &
        \asymp \begin{cases}
                    r^{d\gamma}  (\log r)^\gamma h(R) & \text{ on } \{r \leq R (\log R)^{\frac{-1}{d(1-\gamma)}}\},\vspace{0.3mm}\\
                    r^d  g(R) \log R  & \text{ on } \{R (\log R)^{\frac{-1}{d(1-\gamma)}} \leq r \leq R/2 \}.
                \end{cases}
    \end{align*}
    When $\delta=1/\gamma-1>2$, they satisfy
    \begin{align*}
        \pi_{(u),R}^{(2)} 
        &\asymp \begin{cases}
                    u^{-\gamma} h(R) & \text{ on } \{u \geq (\log R)^{-\frac 1 {1- 2\gamma}}\},\vspace{0.3mm}\\
                    u^{-\gamma \delta} g(R) & \text{ on } \{R^{d(\frac 1 {\gamma \delta}-\frac 1 \gamma)} \leq u \leq (\log R)^{-\frac 1 {1- 2\gamma}}\},\vspace{0.3mm}\\
                    1 & \text{ on } \{u\leq R^{d(\frac 1 {\gamma \delta}-\frac 1 \gamma)} \}.
                \end{cases}\\
        \pi_{r,R}^{(2)} 
        &\asymp \begin{cases}
                    r^{\frac d \delta} h(R) & \text{ on } 
                    \{ r^d \leq  (\log R)^{\frac \delta {\delta-1}}\} ,\vspace{0.3mm}\\
                    r^d g(R) & \text{ on } \{(\log R)^{\frac \delta {\delta-1}} \leq r^d  \}.
                \end{cases}
    \end{align*}
\end{corollary}


We conclude this subsection with some more technical results needed in the core arguments of our proofs. This requires us to also look at  the probabilities of having a path from a ball of radius $r>0$ to a given point. For $u\in (0,1]$ and $x \in \R^d$ with $|x|>r$ define
\begin{align*}
	\theta_{r, (x,u)} & :=\P( (x,u) \longleftrightarrow B(0,r)),\\
	\pi_{r, (x,u)} & := \P((x,u) ^{_{\simone}} B(0,r)).
\end{align*}

 \begin{lemma}\label{lem:second_point_fixed}
     Let $\gamma$ and $\delta$ be as in Proposition~\ref{prop:asymp_escape_crossing_probabilities}. Then, the one-edge crossing probability satisfies, for $|x|=R:$
     \begin{align*}
         \pi_{r,(x,u)}\asymp 
         \begin{cases}
         u^{-1} r^d R^{-d/\gamma} \wedge 1  &\text{if }\delta>1/\gamma, \\
         u^{-1} r^d R^{-d \delta}(1+\log_+(uR^{d\delta})) \wedge 1 &\text{if }\delta=1/\gamma, \\
         u^{-\gamma\delta}r^dR^{-d\delta}&\text{if }\delta<1/\gamma.
         \end{cases}
     \end{align*}
 \end{lemma}

We next provide some upper bounds on three-edge crossing probabilities required for the bootstrap arguments below,  showing that certain three-edge crossing probabilities are not of larger order than the
corresponding two-edge crossing probabilities. 
\begin{lemma}\label{lem:integral_edge_crossing}
Let $\gamma$ and $\delta$ be as in Proposition~\ref{prop:asymp_escape_crossing_probabilities}. Then, we have
\begin{align}  
\int_{\R^d} \int_0^1 \mathbb P \big( (0,u) \sim (x,v) \big) 
\,\pi_{(v),R}^{(2)} \, \de v \, \de x &\lesssim \pi_{(u),R}^{(2)}
& \text{ on } \{ u^{-\gamma} h(R) \leq 1\},
\label{eq:integral_edge_crossing_point}\\
\int_{B(0,R)\setminus B(0,2r)} \int_0^1 \pi_{r,(x,v)}\pi_{(v),R}^{(2)} \, \de v \, \de x & \lesssim \pi_{r,R}^{(2)}.
\label{eq:integral_edge_crossing_ball}
\end{align}
\end{lemma}
\medskip

Lemma~\ref{lem:integral_edge_crossing} is proved in Section~\ref{sec:integral_edge_crossing}.

\subsection{Proof of the one and two-edge crossing probability bounds.}\label{subsec:3.4}

In this section we prove the results stated in Section~\ref{sec:state_one_two_crossing_bounds}
except for Lemma~\ref{lem:integral_edge_crossing}, which will be 
proved in Section~\ref{sec:integral_edge_crossing}.
The probability $\mathbb P_{(x,u)}$ and expectation $\mathbb E_{(x,u)}$ refer to the  vertex set $\mathcal P$ with  the point $(x,u)$ added. We include the case \smash{$\gamma=\frac12$} and $\delta=2$ until further notice.%

\subsubsection*{\textbf{Bounds on $\pi_{(u),R}$}}\label{subsec:bd_pi_u_R}
These bounds are the most straightforward. 
Under $\mathbb P_{(0,u)}$
we write $M_{R}$ for the number of connections of the origin to vertices in $B(0,R)^\complement$. We then have $\pi_{(u),R}=\P_{(0,u)}(M_R>0).$ Note that by construction, $M_R$ is Poisson distributed; therefore $\P_{(0,u)}(M_R>0)\asymp\E_{(0,u)}[M_R]\wedge 1$ and hence it is enough to compute the latter.  By the construction of the graph and Campbell's formula~\cite{Last_Penrose_2017}, we have 
\[\E_{(0,u)}[M_R]=\lambda \int_{B(0,R)^\complement} \int_0^1 \rho(|x|^d (uv)^\gamma) \de v \de x.\]
In the case $\gamma=0$ (which excludes the case $\delta=\infty$) we obtain
\begin{equation}\label{eq:M_R_gamma0}\E_{(0,u)}[M_R]=\lambda \int_{B(0,R)^\complement} \rho(|x|^d ) \de x\asymp R^{d(1-\delta)},\end{equation}
while in the case $\gamma>0$ we can use use the change of variables $|x|=z$ and $w=uvz^{d/\gamma}$ to obtain
\[\E_{(0,u)}[M_R]\asymp \int_R^\infty u^{-1}z^{-1+d(1-\frac{1}{\gamma})} \int_0^{uz^{d/\gamma}} \rho(w^\gamma) \de w \de z.\]
To study this integral, suppose first that $u\leq R^{-d/\gamma}$ and observe that since $\rho\asymp1$ on $(0,1]$, 
we have 
\[\E_{(0,u)}[M_R]\gtrsim\int_R^{u^{-\gamma/d}} u^{-1}z^{-1+d(1-\frac{1}{\gamma})} \int_0^{uz^{d/\gamma}} \rho(w^\gamma) \de w \de z\asymp\int_R^{u^{-\gamma/d}}z^{d-1}\de z\gtrsim1,\]
hence $\pi_{(u),R}\asymp 1$. Suppose next that $u\geq R^{-d/\gamma}$, so that for any $z\geq R$,
\[\int_0^{uz^{d/\gamma}} \rho(w^\gamma) \de w\asymp\begin{cases} 1 &\text{ if }\gamma \delta>1,\\
1+\log(uz^{d\delta}) &  \text{ if } \gamma \delta=1,\\  (uz^{d/\gamma})^{1-\gamma\delta}&\text{ if }\gamma \delta<1.
            \end{cases}\]
Note that this is also valid for the case $\delta=\infty$. Replacing this expression into our previous computation of $\E_{(0,u)}[M_R]$ yields
\begin{equation}\label{eq:M_R_gammapos}
    \E_{(0,u)}[M_R]\wedge1\asymp
            \begin{cases}
                u^{-1} R^{d(1-\frac 1 \gamma)}\wedge1  &\text{ if }\gamma \delta>1,\\
u^{-1} R^{d(1-\delta)} (1+\log_+(uR^{d\delta}) )\wedge1
&  \text{ if } \gamma \delta=1,\\              
                u^{-\gamma \delta} R^{d(1-\delta)}\wedge1 &\text{ if }\gamma \delta<1.
            \end{cases}
\end{equation}

Observe that when $\gamma = 0$, equation~\eqref{eq:M_R_gammapos} reduces to~\eqref{eq:M_R_gamma0}. Setting $u = 1$ in~\eqref{eq:M_R_gammapos} yields the asymptotic behaviour of $g(R)$, and replacing this expression in~\eqref{eq:M_R_gammapos} leads to~\eqref{asymp:1e-escape-probabilities}. These computations are straightforward, except for a minor issue in the case $\gamma\delta = 1$, the function $g(R)$ includes a factor of $\log(R)$, while the corresponding factor in~\eqref{eq:M_R_gammapos} is $1 + \log_+(uR^{d\delta})$. This discrepancy is irrelevant when $u \leq R^{d(1-\delta)}$, since in that regime we have $\pi_{(u),R} \asymp 1$. On the other hand, if $u \geq R^{d(1-\delta)}$, then $R^d \leq uR^{d\delta} \leq R^{d\delta}$, so we indeed obtain $1 + \log_+(uR^{d\delta}) \asymp \log(R)$, as required.

\subsubsection*{\textbf{Bounds on $\pi_{r,(x,u)}$}}
To obtain bounds for $\pi_{r,(x,u)}$ we follow the same program as the one in the previous section, by writing $M_R'$ for the number of connections of $(x,u)$ to vertices in $B(0,r)$, so that $\P_{(x,u)}(M_R'>0)\asymp\E_{(x,u)}[M_R']\wedge 1$. Using again Campbell's formula we arrive at
\begin{align*}
    \E_{(x,u)}[M_R']&=\lambda \int_{B(0,r)} \int_0^1 \rho(|x-y|^d (uv)^\gamma) \de v \de y\\
    &\asymp \int_{B(0,r)} \int_0^1 \rho(R^d (uv)^\gamma) \de v \de y\asymp r^d\int_0^1 \rho(R^d (uv)^\gamma) \de v,
\end{align*}
where we have used that $2r\leq R$ on any domain (as defined in~\eqref{def:domain_D}), so any vertex $(y,v)\in B(0,r)\times[0,1]$ satisfies $|x-y|\asymp R$. The remaining calculations are analogous to the ones given in the previous section, so we omit the details.

\subsubsection*{\textbf{Bounds on $\pi_{r,R}$}}
\label{subsec:bd_r_R}

For given $r$ and $R\ge 2r,$ we observe that any direct connection from $B(0,r)$ to $B(0,R)^\complement$ is of length at least $R/2$. Therefore, we can bound $\pi_{r,R}$ by the expected number of vertices in $B(0,r)$ incident to such a long edge, namely
\[
\pi_{r,R}\le {\mathrm{Vol}}(B(0,r)) \int_0^1 \pi_{(u),R/2}\, \de u \asymp
\begin{cases}
    r^d R^{d(1-\frac 1 \gamma)} \log R &\text{ if }\gamma\delta>1,\\
r^d R^{d(1-\delta)} (\log R)^2 &\text{ if }\gamma\delta=1,\\   
    r^d R^{d(1-\delta)}  &\text{ if }\gamma\delta<1.
\end{cases}
\]
To obtain matching lower bounds in the case $0\leq \gamma\delta<1$ we work conditionally on the locations and weights of all vertices in $B(0,r)$ and in $B(0,2R)\setminus B(0,R)$. With  probability of order one, we find order $r^d$ such vertices in the first set, and order $R^d$ such vertices in the second set. Now each of the $r^dR^d$ possible connections hold independently with probability of order $R^{-d\delta}$, even in the worst case scenario where the vertices have weights one and are at distance $2R+r$ apart, so indeed we can find a connection with probability of order $r^d R^{d(1-\delta)}$.
This analysis covers the particular case where $\delta=2$ and $\gamma<\frac{1}{2}$, giving a lower bound of order \smash{$(r/R)^{d}$} which is of order $1$ on $\{R=2r\}$, thus proving Lemma~\ref{lem:quant}.
\pagebreak[3]
\smallskip

In the case $1<\gamma\delta\leq\infty$ we bound $\pi_{r,R}$ from below by the probability of the event where the strongest vertex in $B(0,r)$ is connected to the strongest vertex in $B(0,2R) \setminus B(0,R)$. Call $U$ and $V$ their respective marks, and observe that they are connected, {conditionally on $U$ and $V$} with probability  of order $(UV)^{-\gamma\delta}R^{-d\delta}\wedge 1$ which is of order $1$ as soon as $UV<R^{-d/\gamma}$. To obtain a lower bound for the probability of this latter event, let $N_r$ and $N_R$ denote the number of vertices in $B(0,r)$ and $B(0,R)^{\complement}$, respectively. Since these variables are Poisson distributed with mean $\E N_r\asymp r^d$ and $\E N_R\asymp R^d$, with probability of order one we have $cr^{d}<N_r< Cr^{d}$ and $cR^{d}<N_R<CR^{d}$ for some constants $c$ and $C$. Assume that this event holds, and further restrict the condition $UV<R^{-d/\gamma}$ to
\[\{r^dR^{-d/\gamma}\leq V\leq R^{-d},\,R^{-d/\gamma}\leq U\leq R^{-d/\gamma}/V\}.\]
It can be checked that from the bounds of $N_r$ and $N_R$, on this set $U$ and $V$ have density of order $r^dR^d$, so we arrive at
\[\P(UV<R^{-d/\gamma})\gtrsim \int_{r^dR^{-d/\gamma}}^{R^{-d}}\int_{R^{-d/\gamma}}^{R^{-d/\gamma}/v}r^dR^d\de u\de v\asymp r^d R^{d(1-1/\gamma)}\log (R^{1/\gamma-1}/r),\]
which matches the upper bound in the case $\gamma\delta>1$ and $\gamma<\frac{1}{2}$, since $\log(R^{1/\gamma-1}/r)\asymp\log R$. Observe that if $\gamma=\frac{1}{2}$ this expression becomes $(r/R)^d\log(R/r)$, which is of constant order on $\{R=2r\}$, thus proving Lemma~\ref{lem:quant} in the case $\gamma=\frac{1}{2}$.
\smallskip

In the case $\gamma\delta=1$ we still consider the strongest vertices in $B(0,r)$ and in $B(0,2R)\setminus B(0,R)$, which are connected with conditional probability of order $(UV)^{-1}R^{-d\delta}$, however, working on the event $\{UV<R^{-2d}\}$ does not give a matching lower bound anymore. Instead, we compute $\E[(UV)^{-1}R^{-d\delta}]$ which, when restricted to the event
\[
\{R^{-d(\frac{\delta-1}2)}r^{-d/2}\le U\le r^{-d}\}\cap \{ R^{-d(\frac{\delta+1}2)}r^{d/2}\le V\le R^{-d}\},
\]
gives a lower bound of order 
$
R^{-d\delta} r^dR^d\log(R^{\delta-1}/r)^2.
$
Since the case $\delta=2$ was already treated, we can assume that $\delta>2$, so the bound becomes of order \smash{$r^dR^{-d(\delta-1)}  (\log R)^2$}, thus matching the upper bound.
\smallskip

Note that this analysis not only completes the proof of~\eqref{asymp:1e-crossing-probabilites}, but also of Lemma~\ref{lem:quant}, which shows that for $\delta=2$ or $\gamma=\frac{1}{2}$ there is no subcritical quantitative phase. We will henceforth assume that $\delta>2$ and \smash{$\gamma<\frac12$}.

\subsubsection*{\textbf{Bounds on \smash{$\pi_{(u),R}^{(2)}$}}}

To attain an upper bound we begin by observing that in order to have a two-edge connection at distance $R$, the vertex $(0,u)$ must either be incident to an edge of length at least $R/2$, or have a neighbour within $B(0,R/2)$ which is incident to such an edge. We write $N_R$ for the number of such neighbours, that is,
\[
N_R=\#\{\x=(x,v) \in \poisson: \x\sim(0,u),\,|x|\leq R/2, \,\x\simone B(\x,R/2)^\complement\},
\]
so we obtain
    $\pi_{^{(u),R}}^{_{(2)}}\le \pi_{(u),R/2}+ \P_{(0,u)}(N_R>0)$.
Similarly, we attain a lower bound by observing that in order to have a two-edge connection at distance $R$, it is enough to find a neighbour of $(0,u)$ within $B(0,R)$ incident to an edge of length $2R$ or more. Defining
\[
N'_R=\#\{\x=(x,v) \in \poisson: \x\sim(0,u),\,|x|\leq R, \,\x\simone B(\x,2R)^\complement\}
\]
we can write the analogous lower bound
\begin{align}\label{eq:lower2esc}
    \pi_{(u),R}^{(2)}&\ge \pi_{(u),R}+ \P_{(0,u)}(N'_R>0).
\end{align}
By Lemma~\ref{lem:renormalization_constants} we have \smash{$\pi_{(u),R/2}\asymp \pi_{(u),R}$}, so it remains to provide an upper bound for $\P_{(0,u)}(N_R>0)$ and a lower bound for $\P_{(0,u)}(N'_R>0)$. For the lower bound, 
we introduce a set $A_{(u),R}\subset B(0,R)\times (0,1]$ such that, if it contains any vertex $\x=(x,v)\in \poisson$, then with probability of order one it belongs to $N'_R$. More precisely, the vertex $\x$ should be sufficiently strong to ensure $\P((0,u)\sim (x,v))\asymp 1$ as well as $\pi_{(v),2R}\asymp 1$, thus giving
\begin{equation}\label{eq:lowersetA}
\P_{(0,u)}(N'_R>0)\gtrsim\P(\poisson \cap A_{(u),R}\ne \emptyset)\asymp \mathrm{Vol}(A_{(u),R})\wedge 1.
\end{equation}
The expectation of $N_R$ provides an upper bound for $\P_{(0,u)}(N_R>0)$ that we can compute  by applying Mecke's formula,
\begin{align}
    \E_{(0,u)}[N_R]&
    \label{int:centre_edge_longedge}
    = \int_{\R^d}\int_0^1 \P((0,u)\sim (x,v))\pi_{(v),R/2} \de v \de x\\
    &\asymp \int_{\R^d}\int_0^1 (u^{-\gamma \delta} v^{-\gamma \delta}|x|^{-d\delta} \wedge 1)\pi_{(v),R/2} \de v \de x 
    \asymp u^{-\gamma} \int_0^1 v^{-\gamma} \pi_{(v),R/2}\ \de v.
    \label{eq:int:N_R} 
\end{align}
Note that \eqref{eq:int:N_R} is also valid when $\gamma=0$ and when $\delta=\infty$. We now provide the derivation of the value of this integral
along with the corresponding definition of the set $A_{(u),R}$.
We obtain asymptotic expressions that yield both the asymptotics for $h(R) = \pi^{_{(2)}}_{^{(1),R}}$ given in Table~\ref{tab:gandh} (by taking $u = 1$), and those in~\eqref{asymp:2e-escape-probabilities}.\pagebreak[3]\medskip

\noindent 
$\bullet$ In the case $\gamma \delta>1$ we have $\pi_{(v), R} \asymp v^{-1}R^{d(1 - 1/\gamma)}\wedge1$, so
\begin{align*}
    \E_{(0,u)}[N_R]&\asymp u^{-\gamma} \int_0^{R^{d(1-1/\gamma)}} v^{-\gamma} \, \de v+ u^{-\gamma} R^{d(1-1/\gamma)}\int_{R^{d(1-1/\gamma)}}^{\infty} v^{-1-\gamma} \de v
    \asymp  u^{-\gamma} R^{d(1-\gamma)(1-\frac 1 \gamma)}.
\end{align*}
This concludes the upper bound in~\eqref{asymp:2e-escape-probabilities}, as $h(R) \asymp R^{d(1-\gamma)(1-\frac 1 \gamma)}$.
Note that $u^{-\gamma}h(r)\lesssim\pi_{(u),R}$ on the domain $\{u\leq R^{-d}\}$, so a matching lower bound can be obtained by neglecting $\P_{(0,u)}(N'_R>0)$ in~\eqref{eq:lower2esc}. To obtain a matching lower bound on $\{u\geq R^{-d}\}$ we define 
\[
A_{(u),R}:= B(0, u^{-\gamma/d}R^{1-\gamma})\times \big(0,R^{d(1-1/\gamma)}\big).
\]
which is contained in $B(0,R)\times(0,1]$, and satisfies 
$\mathrm{Vol}(A_{(u),R})\asymp u^{-\gamma}h(R)$. It can be checked directly that both $\P((0,u)\sim (x,v))\asymp 1$ and $\pi_{(v),2R}\asymp 1$ for any $(x,v)\in A_{(u),R}$ so that~\eqref{eq:lowersetA} holds. Note the definition of the set $A_{(u),R}$ is directly inspired by looking at which domain provides the main contribution to the integral~\eqref{eq:int:N_R}.
\smallskip

\noindent
$\bullet$ The case $\gamma \delta=1$ is treated analogously, where now
$\pi_{(v), R} \asymp v^{-1}R^{d(1 - 1/\gamma)}\log(R)\wedge1$ so that similar computations give
\begin{align*}
    \E_{(0,u)}[N_R]&\asymp u^{-\gamma} R^{d(1-\gamma)(1-\delta)} \log^{1-\gamma} R.
\end{align*}
The lower bound is obtained by following the same approach as in the case $\gamma\delta > 1$, splitting the analysis into the regions $\{u \leq R^{-d}\}$ and $\{u \geq R^{-d}\}$, and applying~\eqref{eq:lowersetA} in the latter. In this case, the set $A_{(u),R}\subseteq B(0,R)\times(0,1]$ is defined as
\[
A_{(u),R}:= B(0, u^{-\gamma/d}R^{1-\gamma}\log^{-\gamma/d} R)\times \big(0,R^{d(1-\delta)} \log R\big).
\] 

\noindent
$\bullet$ In the case $\gamma \delta<1$ we have $\pi_{(v), R} \asymp v^{-\gamma\delta}R^{d(1 - \delta)}\wedge1$, therefore,
\begin{align*}
    \int_0^1 v^{-\gamma} \pi_{(v),R/2} \de v 
    &\asymp \int_0^{R^{d(\frac 1 {\gamma \delta}-\frac 1 \gamma)}} v^{-\gamma}\de v
            + R^{d(1-\delta)} \int_{R^{d(\frac 1 {\gamma \delta}-\frac 1 \gamma)}}^1 v^{-\gamma-\gamma \delta}\de v\\
    &\asymp \left.
    \begin{cases}
        R^{d(1+\frac 1 {\gamma \delta}-\frac 1 \gamma -\frac 1 \delta)} &\text{ if } \delta>\frac 1 \gamma -1\\
        R^{d(1-\delta)}\log R &\text{ if } \delta=\frac 1 \gamma -1\\
        R^{d(1-\delta)} &\text{ if } \delta<\frac 1 \gamma -1
    \end{cases} \right\}
    \asymp h(R).
\end{align*}
Inserting this expression into~\eqref{eq:int:N_R} yields the desired upper bound. 
It remains to find matching lower bounds, which we do by considering each sub-case separately.\pagebreak[3]
\smallskip

\noindent
$\blacktriangleright$ In the sub-case $\delta<\frac 1 \gamma -1$ we have $u^{-\gamma\delta}h(R)\lesssim \pi_{(u),R}$ so we obtain a lower bound by simply neglecting the term $\P_{(0,u)}(N'_R>0)$ in~\eqref{eq:lower2esc}. 
\smallskip

\noindent
$\blacktriangleright$ In the sub-case $\frac 1 \gamma -1<\delta<\frac 1 \gamma$ the lower bound is obtained by following again the same approach as in the case $\gamma\delta > 1$, this time splitting the analysis into the regions $\{u \leq R^{-d/(\gamma\delta)}\}$ and $\{u \geq R^{-d/(\gamma\delta)}\}$. In the former domain, it can be checked that $\pi_{(u),R}\gtrsim u^{-\gamma}h(R)$, whereas in the latter, the lower bounds follow from~\eqref{eq:lowersetA}, with 
\[
A_{(u),R}:= B(0, u^{-\gamma/d}R^{1-\frac 1 \delta}) \times (0,R^{d(\frac 1 {\gamma \delta}-\frac 1 \gamma)}).
\]
\noindent
$\blacktriangleright$ The lower bound in the final sub-case, $\delta=1/\gamma -1$, requires a slightly more delicate argument than the previous ones. First, note that $\pi_{(u),R}\gtrsim u^{-\gamma}h(R)$ on \smash{$\{u\leq(\log R)^{-\frac 1 {1- 2\gamma}}\}$} so on this domain we obtain a matching lower bound by simply neglecting the term $\P_{(0,u)}(N'_R>0)$ in~\eqref{eq:lower2esc}. To obtain a matching lower bound on \smash{$\{u\geq(\log R)^{-\frac 1 {1- 2\gamma}}\}$} we abbreviate \smash{$z_0=(u^{-\gamma\delta}R^{d(\delta-1)})^{\frac{1}{d\delta}}/2$} and define the set
\[
A_{(u),R}:= \{(x,v):R^{\frac{1}{2\delta}}\leq |x|\leq z_0,\text{ and }\,R^{\frac{d(1-\delta)}{\gamma\delta}}\leq v\leq u^{-1}|x|^{-\frac{d}{\gamma}}\}.
\]
Observe that Vol$(A_{(u),R})\asymp u^{-1}R^{-d/2}$. As $u\geq (\log R)^{-\frac 1 {1- 2\gamma}}$ we get that $u^{-1}R^{-d/2}\lesssim 1$, and hence $\mathbb P(|\poisson\cap A_{(u),R}|=1)\asymp \text{Vol}(A_{(u),R})$. We work on this event and denote by $(X,V)$ the unique vertex in $A_{(u),R}$, whose location is uniformly distributed within this set. Conditional on the location of $(X,V)$, we have \[\pi_{(V),2R}\asymp V^{-\gamma\delta}R^{d(1-\delta)}\quad\text{and}\quad\P((0,u)\sim(X,V))\asymp 1.\]
Since $-\gamma\delta=\gamma-1$, the unconditional probability of $N'_R>0$ is then lower bounded by
\begin{align*}
    \P\big(N_R'>0\big)&\geq \P\big(|\poisson\cap A_{(u),R}|=1\big)\iint_{A_{(u),R}}\frac{1}{\text{Vol}(A_{(u),R})}\P((0,u)\sim(x,v))\pi_{(v),2R}\de v\de x\\&\asymp \iint_{A_{(u),R}}v^{\gamma-1}R^{d(1-\delta)}\de v\de x\asymp u^{-\gamma}R^{d(1-\delta)}\log R,
\end{align*}
as wanted.\pagebreak[3]
\medskip

\subsubsection*{\textbf{Bounds on $\pi_{r,R}^{(2)}$}}

These bounds are relatively similar to the bounds of \smash{$\pi_{(u),R}^{(2)}$}, so we treat them more briefly. We provide asymptotics available on the domain $\{R\ge 8r\}$, while we use Lemma~\ref{lem:renormalization_constants} to generalize the result to the whole $\mathcal D$. 
In order to have a two-edge connection from $B(0,r)$ to $B(0,R)^\complement$, we must either have a direct connection from $B(0,2r)$ to $B(0,R/2)^\complement$, or we can find an intermediate vertex $x\in B(0,R/2)\setminus B(0,2r)$, connected via a direct edge to both $B(0,r)$ and $B(0,R)^\complement$. Writing $N_R$ for the number of such intermediate vertices, we then have
\begin{align*}
    \pi_{r,R}^{(2)}&\le \pi_{2r,R/2}+ \P(N_R>0)
    \lesssim  \pi_{r,R}+\E[N_R],
\end{align*}
using Lemma~\ref{lem:renormalization_constants} and $R\ge 8r$ to obtain the second line. Besides, we have
\begin{equation}\label{eq:lower2cross}
\pi_{r,R}^{(2)}\ge \pi_{r,R}+ \P(N'_R>0),
\end{equation}
where now
\[
N'_R=\#\{\x=(x,v) \in \poisson, \exists \y,\z\in \poisson, |z|<r,|x|<R, |y-x|\ge 2R, \z\sim \x \sim \y\}.
\]
Using Mecke's formula and observing that the connection to $B(0,R)^\complement$ must be of length at least $R/2$ and of course different from the connection to $B(0,r)$, we have
\begin{equation*}
\E[N_R] \asymp \lambda \int_{B(0,R/2)\setminus B(0,2r)} \int_0^1 \pi_{r,(x,v)} \pi_{(v),R/2} \de v \de x.
\end{equation*}
The computation of this integral is similar to the one from the previous subsection, but requires considering additional cases, so we leave the details to the reader. We obtain here that
\begin{equation}\label{eq:expectNR}
\E[N_R]\asymp
\begin{cases}
            r^{d\gamma} R^{d(1-\gamma)(1-\frac 1 \gamma)} &\text{ if } \gamma\delta>1, \vspace{0.3mm}\\
        r^{d\gamma} R^{d(1-\frac{1}{\gamma})(1-\frac{1}{\delta})} (\log R)^{1-\gamma} (\log r)^\gamma
     &\text{ if } \gamma\delta=1, \vspace{0.3mm}\\
            r^{\frac{d}{\delta}}R^{d(1-\frac{1}{\gamma})(1-\frac{1}{\delta})} & \mbox{ if } 1>\gamma\delta>1-\gamma, \vspace{0.3mm}\\
            r^{\frac{d}{\delta}}R^{d(1-\frac{1}{\gamma})(1-\frac{1}{\delta})}\log(R/r) & \mbox{ if } \gamma\delta=1-\gamma, \vspace{0.3mm}\\
            r^{d(\frac{1}{\gamma\delta}+\delta-\frac{1}{\gamma})}R^{d(1-\delta)} & \mbox{ if } 1-\gamma>\gamma\delta>1/2,\vspace{0.3mm}\\
            r^{d(2-\delta)}R^{d(1-\delta)} \log r & \mbox{ if } \gamma\delta=1/2,\vspace{0.3mm}\\
            r^{d(2-\delta)}R^{d(1-\delta)} & \mbox{ if } 1/2>\gamma\delta.
        \end{cases}
\end{equation}
From this computation, we conclude that $\pi_{r,R} + \mathbb{E}[N_R]$ coincides with the expression in \eqref{asymp:2e-crossing-probabilites} whenever $\delta\geq 1/\gamma-1$. In the borderline case $\delta=1/\gamma-1$, the claimed upper bound follows by bounding the $\log(R/r)$ term in $\mathbb{E}[N_R]$ by $\log R$. Finally, when $\delta < 1/\gamma - 1$, the bound follows from the fact that the expressions obtained for $\mathbb{E}[N_R]$ are of smaller order than $\pi_{r,R}\asymp r^{d} R^{d(1 - \delta)}$.
\smallskip

We now turn to the lower bounds for $\pi_{r,R}^{(2)}$, starting with the observation that in the case $\delta<1/\gamma-1$ the two-edge and the one-edge crossing probabilities coincide, so the lower bounds are obtained by simply ignoring the term $\P(N'_R>0)$ in~\eqref{eq:lower2cross}. When $\delta>1/\gamma-1$, we obtain matching lower bounds by following an analogous approach to that in the previous section, namely, by constructing a set $A_{r,R}$ such that any vertex placed within it is simultaneously connected to $B(0,r)$ and to $B(0,R)^{\complement}$ with probability of order one. Depending on the relation between $\gamma$ and $\delta$, the set $A_{r,R}$ is defined as
\[
    \begin{cases}
    B(0,r^\gamma R^{1-\gamma})\setminus B(0,r) \times (0,R^{d(1-1/\gamma)}) & \text{on } \{r\le R (\log R)^{-\frac 1 {d(1-\gamma)}}\}  \\ & \text{if } \delta>\frac1\gamma , \\
    B\big(0,r^\gamma R^{1-\gamma}(\frac{\log r}{\log R})^{\gamma/d}\big)\setminus B(0,r) \times (0,R^{d(1-\delta)}\log R)& \text{on } \{r\le R (\log R)^{-\frac 1 {d(1-\gamma)}}\}\\ &   \text{if } \delta=\frac1\gamma , \\
    \left(B\big(0,r^{\frac 1 \delta} R^{1-\frac 1\delta}\big)\setminus B(0,r) \right)\times \big(0,R^{d(\frac 1 {\gamma \delta}-\frac 1 \gamma)}\big) & \text{on } \{r\le R^{1+\delta-\frac 1 \gamma}\} \\ & \text{if } \frac {1} \gamma-1< \delta<\frac 1 \gamma.
    \end{cases}
\]

Finally, the case $\delta=1/\gamma-1$ only needs to be analyzed on the domain $\{ r^d \leq  (\log R)^{\frac \delta {\delta-1}}\}$, using the same approach as in the previous section for this specific case. This time we consider the set
\[
A_{r,R}:= \{(x,v):r\leq|x|\leq (\tfrac{r}{R})^{\frac{1}{\delta}}R,\text{ and }\,R^{d(1-\delta)}\leq v^{\gamma\delta}\leq r^d|x|^{-d\delta}\},
\]
and observe that Vol$(A_{r,R})\asymp r^{d(\gamma\delta-\delta+1)/(\gamma\delta)}$. Since \smash{$\gamma=\frac{1}{\delta+1}$} in this case, it follows that the exponent is negative, and hence \smash{$r^{d(\gamma\delta-\delta+1)/(\gamma\delta)}\lesssim 1$}, so $\mathbb P(|\poisson\cap A_{r,R}|=1)\asymp \text{Vol}(A_{r,R})$. As in the previous section, we denote by $(X,V)$ the unique vertex in $A_{r,R}$, whose location is uniformly distributed within this set. Conditional on $(X,V)$, we have \[\pi_{(V),2R}\asymp V^{-\gamma\delta}R^{d(1-\delta)}\quad\text{and}\quad\pi_{r,(X,V)}\asymp 1,\]
so the unconditional probability of $N'_R>0$ is then lower bounded by
\begin{align*}
    \P\big(N_R'>0\big)&\geq \P\big(|\poisson\cap A_{r,R}|=1\big)\iint_{A_{(u),R}}\frac{1}{\text{Vol}(A_{r,R})}\pi_{r,(x,v)}\pi_{(v),2R}\de v\de x\\&\asymp \iint_{A_{r,R}}v^{\gamma-1}R^{d(1-\delta)}\de v\de x\asymp r^{\frac{d}{\delta}}R^{d(1-\delta)}\log R,
\end{align*}
thus giving the desired lower bound.

\subsection{Proof of Lemma~\ref{lem:integral_edge_crossing}}
\label{sec:integral_edge_crossing}
We begin the proof of  \eqref{eq:integral_edge_crossing_point} by recalling from \eqref{asymp:2e-escape-probabilities} that $\pi^{_{(2)}}_{^{(u),R}}\asymp \pi_{(u),R}+(u^{-\gamma}h(R)\wedge1)$. Looking at the expressions~\eqref{int:centre_edge_longedge} and~\eqref{eq:int:N_R} and their computation already made in the bound of $\pi_{^{(u),R}}^{_{(2)}}$, we deduce both
\begin{align*}
\int_{\R^d} \int_0^1  \mathbb P \big( (0,u) \sim (x,v) \big)   \,\pi_{(v),R} \, \de v \, \de x&\asymp u^{-\gamma}\int_0^1 v^{-\gamma} \,\pi_{(v),R} \, \de v 
\asymp u^{-\gamma} h(R)\\
&\lesssim  \pi_{(u),R}^{(2)}
\phantom{fhojrtt} \text{ on } \{ u^{-\gamma} h(R) \leq 1\},
\end{align*}
and
\begin{align*}
\int_{\R^d} \int_0^1  \mathbb P \big( (0,u) \sim (x,v) \big)   \, \big(v^{-\gamma} h(R)\wedge 1\big) \, \de v \, \de x &\asymp u^{-\gamma}\int_0^1 v^{-\gamma} \,\big(v^{-\gamma} h(R)\wedge 1\big)\, \de v \\
&\lesssim u^{-\gamma} h(R) \int_0^1 v^{-2\gamma} \de v\\
&\lesssim  \pi_{(u),R}^{(2)}
\phantom{fhojrtt} \text{ on } \{ u^{-\gamma} h(R) \leq 1\}.
\end{align*}
Combining the two bounds completes the proof of~\eqref{eq:integral_edge_crossing_point}. 
For the proof of \eqref{eq:integral_edge_crossing_ball} we once again use that $\pi^{_{(2)}}_{^{(u),R}}\asymp \pi_{(u),R}+(u^{-\gamma}h(R)\wedge1)$ and recall from \eqref{eq:expectNR} that
$$\int_{B(0,R/2)\setminus B(0,2r)} \int_0^1 
\pi_{r,(x,v)}\pi_{(v),R}
\, \de v \, \de x \asymp \mathbb EN_R\lesssim  \pi_{r,R}^{(2)},$$
where
$N_R=\{ (x,v)\in \mathcal P \colon x\in B(0,R/2)\setminus B(0,2r) 
\text{ with } B(0,r) \sim (x,v) \sim B(x,R) \}.$
It remains to prove that
\begin{equation}\label{eq:ineqhR}\int_{B(0,R/2)\setminus B(0,2r)} \int_0^1 
\pi_{r,(x,v)}(v^{-\gamma}h(R)\wedge1)
\, \de v \, \de x \lesssim  \pi_{r,R}^{(2)},\end{equation}
which we do by considering different cases for $\gamma$ and $\delta$.\smallskip

\noindent
$\bullet$ If $\delta<\frac1\gamma-1$ we have
$\pi^{(2)}_{(v),R} \asymp \pi_{(v),R}$ and 
\eqref{eq:integral_edge_crossing_ball} is proved. \smallskip

\noindent
$\bullet$ If $\frac1\gamma-1\leq \delta<\frac1\gamma$ we have
\begin{align*}
   \pi_{r,(x,v)} &\asymp v^{-\gamma\delta}r^d|x|^{-d\delta}\wedge1, \quad
    \pi_{r,R}^{(2)} \asymp \pi_{r,R}+ r^{\frac d \delta} h(R),
\end{align*}
so using the change of variables $z=v^{\frac{\gamma}{d}}r^{-\frac{1}{\delta}}x$ we can bound the left-hand side of \eqref{eq:ineqhR} by
\[
\int_{\R^d}\int _0^1(v^{-\gamma\delta}r^d|x|^{-d\delta}\wedge1)v^{-\gamma}h(R) \de v \, \de x\asymp r^{\frac{d}{\delta}}h(R)\int_0^1v^{-2\gamma} \de v \int_{\R^d}(|z|^{-d\delta}\wedge1)\de z\asymp r^{\frac{d}{\delta}}h(R).\]
Since $r^{\frac{d}{\delta}}h(R)\lesssim \pi_{r,R}^{(2)}$, \eqref{eq:integral_edge_crossing_ball} is proved.
\smallskip

\noindent
$\bullet$ If $\delta=\frac1\gamma$ we have
\begin{align*}
   \pi_{r,(x,v)} &\asymp v^{-\gamma\delta}r^d|x|^{-d\delta}(1+\log_{+}(v|x|^{d\delta}))\wedge1, \quad
    \pi_{r,R}^{(2)} \asymp \pi_{r,R}+ r^{d\gamma}(\log r)^{\gamma} h(R).
\end{align*}
The proof of \eqref{eq:integral_edge_crossing_ball} is obtained by applying the same change of variables as in the previous case. The upper bound follows from analogous computations and the observation that
\[
\int_{\R^d}(|z|^{-d\delta}(1+\log_{+}(r^d|z|^{d\delta}))\wedge1)\de z\asymp(\log r)^{\gamma}.\]

\smallskip
\noindent
$\bullet$ If $\delta>\frac1\gamma$ we have 
\begin{align*}
   \pi_{r,(x,v)} &\asymp v^{-1}r^d|x|^{-d/\gamma}\wedge1, \quad
    \pi_{r,R}^{(2)} \asymp \pi_{r,R}+ r^{d\gamma} h(R),
\end{align*}
and the proof follows by applying the change of variables $z=v^{\frac{\gamma}{d}}r^{-\gamma}x$ and performing the same computations from the previous cases, which we omit.

\section{Bootstrapping: The key argument}\label{sec:bootstrap}

In this section we provide the key ingredient behind our main result: the upper bounds in Theorem~\ref{thm:two-edges}.  In Section~\ref{subsec:naive_upper_bound} we observe that from Proposition~\ref{prop:annulus_r_2r} we get upper bounds on the annulus crossing and escape probabilities, which imply the bounds in Theorem~\ref{thm:two-edges} on a domain where $r$ is within a constant multiple of $R$ and $u$ is small. In Section~\ref{subsec:geometric_argument}, we present the crucial geometric argument, Proposition~\ref{lem:partition_upper_bound}, which 
allows us to iteratively improve upper bounds, such as those in Section~\ref{subsec:naive_upper_bound}. In Section~\ref{subsec:bootstrap_argument}, we use this as a bootstrap step to recursively enlarge the domain on which Theorem~\ref{thm:two-edges} holds. The proof will be completed 
in Section~\ref{sec:thms} when we use the calculations of Section~\ref{sec:one_two_crossing_bounds} to show that
any domain of the form $\{r\geq R^{\varepsilon}\}$, resp.\ $\{u\leq R^{-\varepsilon}\}$, is covered by the domains generated in finitely many steps of this procedure.

\subsection{The easy bounds} \label{subsec:naive_upper_bound}

First, note the trivial lower bounds 
\begin{align*} 
    \theta_{(u),R} \geq \pi_{(u), R}^{(2)} \quad\text{ and } \quad  \theta_{r,R} \geq \pi_{r,R}^{(2)},
\end{align*}
which hold since restricting paths to have a certain number of edges decreases their probability. These lower bounds hold for all choices of $\delta>2$ and $0\leq\gamma< 1/2$, directly implying the lower bounds in Theorem~\ref{thm:two-edges}. The rest of this section focuses on proving non-trivial upper bounds.\medskip


We recall that while the computations of Section~\ref{sec:one_two_crossing_bounds} were available for all $\lambda>0$, we now assume that we are in the quantitatively subcritical phase, that is $0< \lambda < \widehat \lambda_c$. In particular, in this case we can make use of Proposition~\ref{prop:annulus_r_2r}. 
Recalling also~\eqref{eq:annulus_r_2r_crossing_with_one_or_two_edges}, this proposition is the statement of \eqref{eq:asymp_theta_r_R} in Theorem~\ref{thm:two-edges}, but restricted to the domain $\{\tfrac{R}{2}=r\}$, namely
\begin{equation}\label{eq:asymp_restricted}\theta_{r,R}\asymp \pi_{^{r,R}}^{_{(2)}}\text{ on }\{ r=\tfrac{R}2\}.\end{equation}
Lemma~\ref{lem:renormalization_constants} allows to extend \eqref{eq:asymp_restricted} to a larger domain, namely
\begin{equation}\label{eq:naive_Dzero}
    \theta_{r,R}\asymp \pi_{r,R}^{(2)} \text{ on } D_0,
\end{equation}
where
\begin{equation}\label{eq:def_Dzero}
    D_0:=\{ \tfrac{R}{32} \le r \le \tfrac{R}{2} \}.
\end{equation}
To see this, recall that $\theta_{r, R}$ and $\pi_{r,R}$ are  increasing in $r$. Thus,  on $D_0$, Proposition~\ref{prop:annulus_r_2r} and Lemma~\ref{lem:renormalization_constants} give
$\theta_{r, R} \leq \theta_{R/2,R} \asymp \pi_{R/2, R} \asymp\pi_{R/32, R}\leq \pi_{r,R} \leq \pi_{^{r,R}}^{_{(2)}}$.
Note that using the same arguments we could  extend~\eqref{eq:asymp_restricted} to domains of the form \smash{$\{ c R\le r \le \frac{R}2\}$} 
with $c\in (0,1/2)$ any constant. Here and in what follows $c=1/32$ is a convenient but arbitrary choice.
\medskip

We now turn our attention to $\theta_{(u),R}$, for which there is no result analogous to Proposition~\ref{prop:annulus_r_2r}. The next lemma shows that this proposition can also be used to obtain a version of \eqref{eq:asymp_theta_u_R} on an explicitly defined domain. 
\begin{lemma}\label{lem:naive}
We have, for all $0<u<1$ and $R\ge 4,$
\begin{equation}%
    \theta_{(u), R}  \leq \pi_{(u), R/2} + \theta_{R/2, R} \label{eq:naive_upperbound_u_R}.
\end{equation}
Consequently, we have
\begin{equation}%
    \theta_{(u), R}  \asymp  \pi_{(u), R}^{(2)}  \text{ on } \widetilde D_0,
    \label{eq:naive_u}
\end{equation}
where the domain $\widetilde D_0$ is defined as
\begin{equation}
 \widetilde D_0:=   \begin{cases}
    \{ u^{-1}\ge R^{d } \log R\}& \text{ if }\delta\ge \frac 1 \gamma,\\
    \{ u^{-\gamma \delta}\ge R^d\}& \text{ if }\frac 1 \gamma>\delta.
\end{cases}\label{def:set_tD1}
\end{equation}
\end{lemma}

\begin{remark}\label{rem:d0_tilde}
    Using the asymptotic expressions given in Proposition~\ref{prop:asymp_escape_crossing_probabilities}, it can be 
    checked that $\widetilde D_0$ is defined precisely such that
    \begin{align}
        \pi_{R}\lesssim \pi_{^{(u), R}}^{_{(2)}} \quad &\text{ on }\widetilde D_0 \nonumber\\
        \pi_{^{(u), R}}^{_{(2)}} \lesssim \pi_R \quad &\text{ on }\widetilde D_0^\complement,\label{ineq:defining_tildeD0_complement}
    \end{align}
    where we recall that $\widetilde D_0^\complement = \mathcal{D}\setminus \widetilde D_0$.
    \end{remark}

\begin{proof}[Proof of Lemma~\ref{lem:naive}.]
Rewriting $\{(0,u)\longleftrightarrow B(0,R)^{\complement}\}$ in terms of $\{(0,u)^{_{\simone}} B(0,R/2)^{\complement} \}$ and its complement, we obtain the inclusion 
\begin{align*}
    \{(0,u)\longleftrightarrow B(0,R)^{\complement}\} \subseteq \{(0,u)^{_{\simone}} B(0,R/2)^{\complement} \} \, \cup\,  \{B(0,R/2)\longleftrightarrow B(0,R)^{\complement}\},
\end{align*}
which gives the upper bound \eqref{eq:naive_upperbound_u_R}. Using Proposition~\ref{prop:annulus_r_2r} gives \smash{$\theta_{(u), R}  \lesssim \pi_{(u), R/2} + \pi_{R}$}, so~\eqref{eq:naive_u} follows from the trivial bound $\pi_{(u), R} \leq \pi_{^{(u), R}}^{_{(2)}}$ and Remark~\ref{rem:d0_tilde}.
\end{proof}

{Note that in the previous lemma, \eqref{eq:naive_upperbound_u_R} holds for all $0<u<1$ and $R\geq 4$, and that by monotonicity, $\theta_{r,R}\leq \theta_{R/2,R}$ holds for any $2\leq r\leq R/2$. By combining these inequalities with Proposition~\ref{prop:annulus_r_2r}, we conclude this section with the following upper bounds, which hold on all $\mathcal{D}$,
\begin{equation}
\label{eq: initial_upper}
    \theta_{(u),R} \lesssim \pi_{(u), R}^{(2)}+\pi_{R} \quad\text{ and } \quad  \theta_{r,R} \lesssim \pi_{r, R}^{(2)}+\pi_{R}.
\end{equation}
}
\subsection{The geometric argument} \label{subsec:geometric_argument}

In this section, we use a geometric argument to obtain the following proposition. 
\begin{proposition}
\label{prop:bootstrapping}
    {Let $\cE\colon(4,\infty)\to [0,\infty)$, $R\mapsto\cE_R$ be a map satisfying $\cE_{R}\lesssim\cE_{2R}$, such that}
    \begin{align}
        \theta_{r,R} & \lesssim \pi_{r,R}^{(2)} + \cE_R \label{eq:assmpt_theta_r_R},\\
        \theta_{(u), R} & \lesssim \pi_{(u), R}^{(2)} + \cE_R,\label{eq:assmpt_theta_u_R}
    \end{align}
    holds on $\mathcal{D}$. Then
     {
     \begin{align}
        \theta_{r,R} &  \lesssim \pi_{r, R}^{(2)} + \pi_r\cE_{R},\label{eq:bootstrap_theta_r_R} \\
        \theta_{(u), R} & \lesssim \pi_{(u), R}^{(2)} +  \pi_{r,R}^{(2)}+  (\pi_{(u),r}+
        \pi_{r}) \cE_{R}.\label{eq:bootstrap_theta_u_R}
    \end{align}}
\end{proposition}

\begin{remark}
  In the previous section we showed that~\eqref{eq:assmpt_theta_r_R},~\eqref{eq:assmpt_theta_u_R} hold with the error term~$\cE_R=\pi_R$. In Section~\ref{subsec:bootstrap_argument} we will use Proposition~\ref{prop:bootstrapping} to obtain these inequalities with a reduced error term at the expense of a larger constant factor hidden in the $\lesssim$ notation. Note that to use the proposition the error term needs to be defined independently of $u,r$.
\end{remark}      
\begin{remark}
    While the right-hand side in~\eqref{eq:bootstrap_theta_u_R} depends on $r$, the
    left-hand side does not. We can therefore 
    choose $r$ suitably in the application of Proposition~\ref{prop:bootstrapping}.
  This is the purpose of the function $s(u,R)$ defined at the beginning of Section~\ref{subsec:bootstrap_argument}.
\end{remark}
\begin{remark}\label{remark: trivial upper bound}
   Note that on domains where \smash{$\pi^{^{(2)}}_{^{(u),R}}\asymp 1$}, Equation~\eqref{eq:bootstrap_theta_u_R} holds trivially. This is particularly the case for the domain $\{u^{-\gamma}h(R/2)\geq 1\}$. In this case the geometric argument is not needed at all.
\end{remark}

The idea of the proof of  Proposition~\ref{prop:bootstrapping} lies in strengthening the upper bounds ~\eqref{eq:assmpt_theta_r_R} and~\eqref{eq:assmpt_theta_u_R}, by splitting the event of having a path from $(0,u)$ to $B(0,R)^{\complement}$ into five carefully chosen events and bounding the probability of each individual event. The next lemma states an upper bound on $\theta_{(u), R}$ and $\theta_{r,R}$ obtained by such a partitioning and is a key tool in proving Proposition~\ref{prop:bootstrapping}. 
\begin{lemma}[Geometric argument]\label{lem:partition_upper_bound}
    For $u\in (0,1]$ and $2\leq r \leq R/32$,
    \begin{align}
         \theta_{r,R} & \leq \theta_{r, 2r}\theta_{8r, R} + \pi_{4r, R/2} + \pi_{r,R/2} + J_1 + J_2,\label{eq:exact_upbd_partition_theta_r}\\
          \theta_{(u), R} & \leq \theta_{(u), 2r}\theta_{8r, R} + \pi_{4r, R/2} + \pi_{2r, R/2} + \pi_{(u), R/2} + J_1 + J_2 + J_3, \label{eq:exact_upbd_partition}
    \end{align}
    and as an immediate consequence,
     \begin{align}
        \theta_{r, R} &\lesssim  \theta_{r}\theta_{8r, R} + \pi_{r, R}  + {J_1 + J_2}, & \label{eq:upbd_partition_theta_r}\\
        \theta_{(u), R} &\lesssim \theta_{(u), {2}r}\theta_{
        8r, R} + \pi_{r, R} + \pi_{(u),R} + J_1 + J_2+J_3, \label{eq:upbd_partition}
    \end{align}
    on $\{r\leq R/32\}$ and where
    \begin{align*}
        J_1 & := \int_{B(0, R/2)\setminus B(0,4r)}\int_{0}^1  \pi_{2r, (y,v)}\theta_{(v), R/2} \, dv \, dy, \\
        {J_2} & :=\int_{B(0,R/2)\setminus B(0,8r)}\int_0^1 \pi_{4r, (y,v)}\theta_{(v), R/2} \, dv \, dy,  \text{ and }\\
        J_3 & := \int_{B(0, R/2)\setminus B(0,4r)}\int_{0}^1  \mathbb P \big( (0,u) \sim (y,v) \big) \theta_{(v), R/2} \,  dv \, dy.
    \end{align*}
\end{lemma}

The integrals appearing in the previous lemma are reminiscent of some of the integrals computed in Section~\ref{sec:one_two_crossing_bounds}, the main difference being that these integrals feature the escape probability $\theta_{(v),R/2}$ instead of $\pi_{(v),R/2}$. In the next lemma, we provide bounds on these integrals, assuming~\eqref{eq:assmpt_theta_u_R} holds for the escape probability. 



\begin{lemma}\label{lem:bds_integrals_J_I} Assume~\eqref{eq:assmpt_theta_u_R} holds. On the domain $\{r\leq R/32\}$,
    \begin{align*}
        J_1 &\lesssim \pi_{r, R}^{(2)} + \pi_{r}{\cE_{R}},\\
        J_2 & \lesssim \pi_{r, R}^{(2)} + \pi_{r}{\cE_{R}}.
    \end{align*}
    On the domain $\{u^{-\gamma}h(R/2) \leq 1,\pi_{(u),r}\leq 1/2\}$,
    \begin{align*}
        J_3 & \lesssim \pi_{(u), R}^{(2)} + 
        \pi_{(u),r} \cE_{R}.
    \end{align*}
\end{lemma}

With these two lemmas at hand we can now prove Proposition~\ref{prop:bootstrapping}. 

\begin{proof}[Proof of Proposition~\ref{prop:bootstrapping}]
    We begin by bounding $\theta_{r,R}$. Note that on 
    {$D_0$} we already have a good bound on $\theta_{r,R}$ by~\eqref{eq:naive_Dzero}, which directly implies \eqref{eq:bootstrap_theta_r_R}. Thus we now work on $\{r\le R/32\}$. 
    Using the upper bound \eqref{eq:upbd_partition_theta_r} given in Lemma~\ref{lem:partition_upper_bound}, 
    together with the bounds on $J_1$ and $J_2$ from Lemma~\ref{lem:bds_integrals_J_I}, we get 
      \begin{align}
        \theta_{r,R} &  \lesssim \theta_r \theta_{8r,R} + \pi_{r,R} + \pi_{r, R}^{(2)} + \pi_r \cE_{R}.
    \end{align}
   Using \eqref{eq:assmpt_theta_r_R} to bound $\theta_{8r,R}$ we obtain that
    \begin{align*}
           \theta_{r,R} & \lesssim \theta_r \pi_{8r,R}^{(2)} + \theta_r \cE_{R} + \pi_{r,R} + \pi_{r, R}^{(2)} + \pi_r {\cE_{R}}.
    \end{align*}
    Recalling by Proposition~\ref{prop:annulus_r_2r} that we have $\theta_r\asymp \pi_r$, we conclude that
    \begin{align*}
         \theta_{r,R} & \lesssim \pi_r \pi_{8r,R}^{(2)}  + \pi_{r,R} + \pi_{r, R}^{(2)} + \pi_r \cE_{R} \lesssim  \pi_{r, R}^{(2)}  + \pi_r  \cE_{R},
    \end{align*} 
    using the fact that $\pi_{8r,R}^{(2)} \asymp \pi_{r,R}^{(2)}$ and $\pi_{r,R}\leq \pi_{r,R}^{(2)}.$ This proves \eqref{eq:bootstrap_theta_r_R}.    \pagebreak[3]\smallskip

    \medskip
 
Next, we bound \smash{$\theta_{(u),R}$}. We begin by proving \eqref{eq:bootstrap_theta_u_R} on various subdomains of $\mathcal{D}$ where the proof is either trivial or follows directly from previous results. 
First, consider the domain~$D_0$. Here, Lemma~\ref{lem:renormalization_constants} gives $\pi_{R}\asymp \pi_{^{r,R}}^{_{(2)}}$, and thus \eqref{eq: initial_upper} implies \smash{$\theta_{(u),R} \lesssim \pi_{^{(u), R}}^{_{(2)}}+\pi_{^{r,R}}^{_{(2)}}$}, from which \eqref{eq:bootstrap_theta_u_R} follows immediately.
Second, on the domain 
$\{\pi_{(u),r}\geq 1/2\}$ we have \smash{$\pi_{(u),r}+\pi_{r}\asymp 1$}, so \eqref{eq:bootstrap_theta_u_R} is directly implied by hypothesis \eqref{eq:assmpt_theta_u_R}.
Third, on  \smash{$\{u^{-\gamma}h(\frac{R}2)\geq 1\}$}, we have \smash{$\pi_{^{(u),R}}^{_{(2)}}\asymp 1$}, and hence \eqref{eq:bootstrap_theta_u_R} holds immediately.
\medskip

We now prove \eqref{eq:bootstrap_theta_u_R} on $\{r\leq R/32\}\cap\{u^{-\gamma}h(R/2) \leq 1\}\cap\{\pi_{(u),r}\leq 1/2\}$, which is the interesting case where the geometric argument is needed. We observe that on this domain we can use both \eqref{eq:upbd_partition} in Lemma~\ref{lem:partition_upper_bound}, and the bounds on $J_1$, $J_2$, and $J_3$ from Lemma~\ref{lem:bds_integrals_J_I}, thus obtaining
    \begin{align}
     \label{eq:intermediate_bound_theta_u_R}
        \theta_{(u),R}&  \lesssim \theta_{(u),{2r}}\theta_{8r,R} + \pi_{r,R}^{(2)} + \pi_{(u), R}^{(2)} + (\pi_{(u), r}+\pi_r)\cE_{R},
    \end{align}
    where we used the fact that \smash{$\pi_{r,R}\leq \pi_{{r,R}}^{_{(2)}}$ and $\pi_{(u),R} \leq \pi_{^{(u),R}}^{_{(2)}}$}. Observe now that \eqref{eq:naive_upperbound_u_R} together with Lemma~\ref{lem:renormalization_constants} and Proposition \ref{prop:annulus_r_2r} 
    gives $\theta_{(u),2r}\lesssim  \pi_{(u),r}+\pi_r$, and \eqref{eq:assmpt_theta_r_R} together with Lemma~\ref{lem:renormalization_constants} gives 
    \smash{$\theta_{8r,R}\lesssim \pi_{^{r,R}}^{_{(2)}} + \cE_R$}. Combining both bounds we arrive at
    \begin{align*}
           \theta_{(u),2r}\theta_{8r,R} & \lesssim \pi_{r,R}^{(2)} + (\pi_r + \pi_{(u),r})\cE_{R}.
    \end{align*}
    Using this bound in~\eqref{eq:intermediate_bound_theta_u_R} concludes the proof of~\eqref{eq:bootstrap_theta_u_R}.\qedhere
\end{proof}

We now prove Lemma~\ref{lem:partition_upper_bound}, which constitutes the geometric core of the proof.


\begin{proof}[Proof of Lemma~\ref{lem:partition_upper_bound}]
{We first observe that \eqref{eq:upbd_partition_theta_r} and \eqref{eq:upbd_partition} follow from \eqref{eq:exact_upbd_partition_theta_r} and \eqref{eq:exact_upbd_partition} by applying Lemma~\ref{lem:renormalization_constants}. We therefore focus on proving these latter inequalities, beginning with \eqref{eq:exact_upbd_partition}.}

\begin{figure}[h]
    \centering    \includegraphics[width=0.85\linewidth]{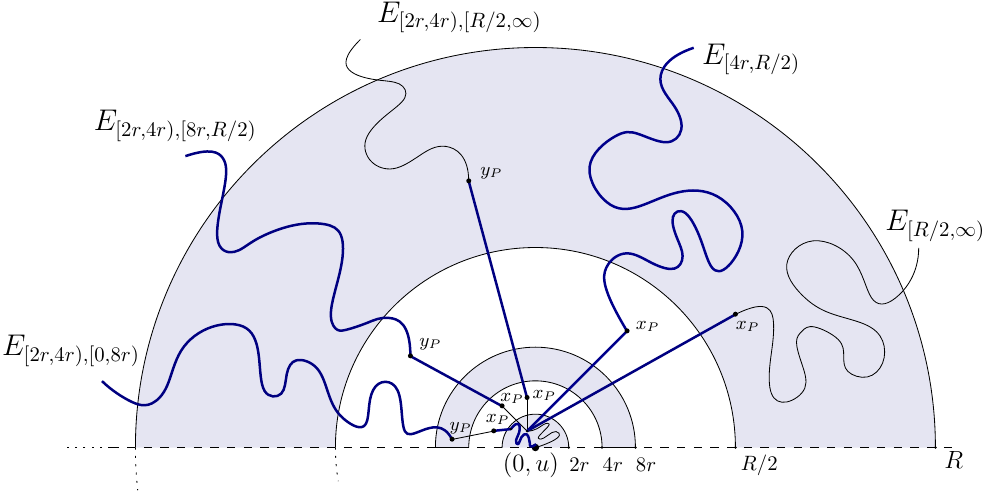}
    \caption{The paths from $(0,u)$ to $B(0,R)^{\complement}$ are split in five disjoint classes. In this figure edges are seen as straight lines and paths as curved lines. The parts contributing to our bounds are drawn in navy and bold. }
    \label{fig:crossings_partition}
\end{figure}

\pagebreak[3]
The idea of the proof consists of splitting the event of having a path from $(0,u)$ to $B(0,R)^{\complement}$ into five events and bounding their respective probabilities. To this end, let $\cS$  be the set of  paths with starting vertex $(0,u)$ ending in $B(0,R)^{\complement}$. 
For $P\in \cS$, write $x_P$ for the first vertex on this path that is outside $B(0,2r)$ and, if $x_P \in B(0,R)$, write $y_P$ for the next vertex on the path. For intervals $I_1,I_2 \subseteq [0,\infty)$, define the event $E_{I_1, I_2}$ as
\begin{align}\label{def:event_E_I_I}
    E_{I_1, I_2} & := \{\exists \,P\in \cS  \text{ such that } |x_P|\in I_1, \, |y_P| \in I_2 \},
\end{align}
and define $E_{I_1}$ analogously, only requiring that $|x_P|\in I_1$.
We can then rewrite the event $\{|\cS| \geq 1
\}$ as a union of  events 
\begin{align}\label{eq:partition_union_events_path}
    \{|\cS| \geq 1
\} & = E_{[2r,4r),[0,8r)} \cup  E_{[2r,4r),[8r, R/2)} \cup  E_{[2r,4r),[R/2,\infty)} \cup  E_{[8r, R/2)} \cup  E_{[R/2, \infty)}.
\end{align}
See Figure~\ref{fig:crossings_partition} for illustration. Note that $\theta_{(u),R} = \P(|\cS|\geq 1)$ and so, we obtain an upper bound on $\theta_{(u),R}$ by bounding the probability of each individual event in \eqref{eq:partition_union_events_path}.
\medskip
\pagebreak[3]

\noindent
$\bullet$ Begin by observing that $E_{[2r,4r), [0,8r)}$ implies the existence of a path from $(0,u)$ to $B(0,4r) \setminus B(0,2r)$ and of a path from $B(0, 8r)$ to $B(0,R)^{\complement}$, where the two paths do not share a common vertex since paths are self-avoiding. Therefore $E_{[2r,4r), [0,8r)}$ is contained in the disjoint occurrence of the events $\{(0,u) \longleftrightarrow B(0,4r) \setminus B(0,2r)\}$ and $\{B(0,8r) \longleftrightarrow B(0,R)^{\complement}\}$. Since both events are increasing, an application of the van den Berg-Kesten (BK) inequality for marked Poisson point processes \cite{vandenBerg_1996} gives
\begin{align}\label{eq:prob_E_le_2r_le_3r}
    \P(E_{[2r,4r), [0,8r)})\leq \theta_{(u),2r}\theta_{8r, R}.
\end{align}

\noindent
$\bullet$ The event $E_{[2r,4r), [R/2, \infty)} $ implies that there is an edge from $B(0,4r)$ to $B(0,R/2)^{\complement}$, thus 
\begin{align}\label{eq:prob_le_2r_ge_R2}
    \P(E_{[2r,4r), [R/2, \infty)}) \leq \pi_{4r, R/2}.
\end{align}

\noindent
$\bullet$ On the event $E_{[R/2, \infty)}$, $x_P$ lies outside $B(0,R/2)$ while the predecessor of $x_P$ in $P$ lies within $B(0,2r)$. Therefore, $E_{[R/2, \infty)}$ implies that there exists an edge from $B(0,2r)$ to $B(0,R/2)^{\complement}$ or an edge from $(0,u)$ to $B(0,R/2)^{\complement}$. It follows that 
\begin{align}\label{eq:prob_ge_R2}
    \P(E_{[R/2, \infty)}) \leq \pi_{2r, R/2} + \pi_{(u),R/2}.
\end{align}

\noindent
$\bullet$ Observe that on the event $E_{[2r,4r), [8r, R/2)}$, there exists a vertex $$(y_P,v) \in B(0, R/2)\setminus B(0,8r) \times (0,1],$$ such that $B(0,4r) ^{_{\simone}}  (y_P,v)$ and $ (y_P,v) \longleftrightarrow B(0,R)^{\complement}$, where the two paths do not share a common vertex. By a first moment method, the probability of $E_{[2r,4r), [8r, R/2)}$ is then upper bounded by the expected number of such vertices $(y_P,v) \in \cP$. Before computing this expectation, observe that given $(y,v) \in \cP \cap (B(0, R/2)\setminus B(0,8r) \times [0,1])$, the probability that two vertex disjoint paths $B(0,4r)^{_{\simone}}  (y,v)$ and $ (y,v) \longleftrightarrow B(0,R)^{\complement}$ exists can be bounded from above by 
\begin{align*}
    \pi_{4r, (y,v)}\theta_{(y,v), R},
\end{align*}
by the BK inequality for marked Poisson point processes \cite{vandenBerg_1996}. Note that we have $\theta_{(y,v), R} \leq \theta_{(v), R/2}$ as $|y|< R/2$. Combining the above together with a first moment method, we deduce that 
\begin{align}\label{eq:integral_J1}
    \P(E_{[2r,4r), [8r, R/2)}) & \leq \int_{(B(0, R/2)\setminus B(0,8r)}\int_{0}^1  \pi_{4r, (y,v)}\theta_{(v), R/2} dvdy = J_2.
\end{align}

\noindent
$\bullet$ Finally, on the event $E_{[4r, R/2)}$ there exists a vertex $$(x_P,v) \in B(0, R/2)\setminus B(0,4r) \times (0,1],$$ such that $ (x_P,v) \longleftrightarrow B(0,R)^{\complement}$, and either $B(0,2r) ^{_{\simone}}  (x_P,v)$, or $(0,u)\sim (x_P,v)$.
Following a similar argument as above and a first moment bound, the probability of $E_{[4r, R/2)}$ can be upper bounded  by 
\begin{align}\label{eq:prob_2r_R2_eq1}
    \P(E_{[4r, R/2)}) & \leq \int_{(B(0, R/2)\setminus B(0,4r)}\int_{0}^1  ( \mathbb P \big( (0,u) \sim (x,v) \big) + \pi_{2r, (x,v)})\theta_{(v), R/2} \, dv \, dx\\
    & = J_1 + J_3.\notag
\end{align}
Combining the above bounds \eqref{eq:prob_E_le_2r_le_3r}, \eqref{eq:prob_le_2r_ge_R2}, \eqref{eq:prob_ge_R2},  \eqref{eq:integral_J1} and \eqref{eq:prob_2r_R2_eq1} together with \eqref{eq:partition_union_events_path}, we obtain 
\begin{align*}
    \theta_{(u), R} \leq \theta_{(u), 2r}\theta_{8r, R} + \pi_{4r, R/2} + \pi_{2r, R/2} + \pi_{(u), R/2} + J_1 + J_2 + J_3.
\end{align*}
This concludes the proof of \eqref{eq:exact_upbd_partition}.\\

The argument for the upper bound on $\theta_{r, R}$ is very similar to the proof above. Let $\cS'$  be the set of paths with starting vertex in $B(0,r)$ and ending in $B(0,R)^\complement$. For $P\in \cS'$, write~$x_P'$ for the first vertex of this path that is outside $B(0,2r)$ and write $y_P'$ for the next vertex on the path. Define the events $E_{I_1,I_2}'$ and $E_{I_1}'$ analogously to the previous case, see \eqref{def:event_E_I_I}, except that the paths are now in $\cS'$. We then have that 
$\theta_{r, R} = \P(|\cS'| \geq 1)$ and 
\begin{align*}
    \{|\cS'| \geq 1\} & =  E_{[2r,4r),[0,8r)}' \cup  E_{[2r,4r),[8r, R/2)}' \cup  E_{[2r,4r),[R/2,\infty)}' \cup  E_{[4r, R/2)}' \cup  E_{[R/2, \infty)}'.
\end{align*}
Using similar arguments to obtain the bounds \eqref{eq:prob_E_le_2r_le_3r}, \eqref{eq:prob_le_2r_ge_R2} and \eqref{eq:prob_ge_R2}, for which we omit the details, we  deduce that 
\begin{align*}
    \P( E_{[2r,4r),[0,8r)}' \cup  E'_{[2r,4r),[R/2,\infty)} \cup  E'_{[R/2, \infty)}) & \leq \theta_{r, 2r}\theta_{8r, R} + \pi_{4r, R/2} + \pi_{r, R/2}.
\end{align*}
Further, adapting the arguments for \eqref{eq:integral_J1} and \eqref{eq:prob_2r_R2_eq1} we obtain 
\begin{align*}
    \P(E_{[2r,4r),[8r, R/2)}') & \leq \int_{B(0,R/2)\setminus B(0,8r)}\int_0^1 \pi_{4r, (y,v)}\theta_{(v), R/2}dvdy = J_2,\\
    \P( E_{[4r, R/2)}') & \leq \int_{B(0,R/2)\setminus B(0,4r)}\int_0^1 \pi_{2r, (x,v)}\theta_{(v), R/2}dvdx = {J_1}.
\end{align*}
Combining all of the above we obtain 
\begin{align*}
    \theta_{r,R} \leq \theta_{r, 2r}\theta_{8r, R} + \pi_{4r, R/2} + \pi_{r, R/2} + J_1 + J_2,
\end{align*}
concluding the proof of \eqref{eq:exact_upbd_partition_theta_r}.\qedhere

\end{proof}

\begin{proof}[Proof of Lemma~\ref{lem:bds_integrals_J_I}]
    Applying the inequality \eqref{eq:assmpt_theta_u_R}, which we assume to hold, to the integrand of $J_1$, we obtain that
    \begin{align}
        J_1 & \lesssim\int_{B(0, R/2)\setminus B(0,4r)}\int_{0}^1  \pi_{2r, (y,v)}\pi^{(2)}_{(v), R/2} dvdy \notag \\
        & \qquad + \int_{B(0, R/2)\setminus B(0,4r)}\int_{0}^1  \pi_{2r, (y,v)}\cE_{R/2} dvdy, \label{eq:upbd_J1}
    \end{align}  
    on the domain $\{r \leq R/32\}$.
    The second summand on the right hand side can be bounded from above by a constant multiple of $\pi_{2r, 4r}\cE_{R/2}$. For the first summand, we apply Lemma~\ref{lem:integral_edge_crossing} \eqref{eq:integral_edge_crossing_ball} to obtain the upper bound 
    \begin{align*}
        \int_{B(0, R/2)\setminus B(0,4r)}\int_{0}^1  \pi_{2r, (y,v)}\pi^{(2)}_{(v), R/2} dvdy & { \, \lesssim \,} \pi_{2r,R/2}^{(2)}.
    \end{align*}
    Plugging this into \eqref{eq:upbd_J1}, using the fact that $\pi_{2r,4r} \asymp \pi_r$ and $\pi_{^{2r, R/2}}^{_{(2)}}\asymp \pi_{r,R}$, and the assumption $\cE_{R/2} \lesssim \cE_{R}$, we get that
         $J_1  \lesssim  \pi_{^{r, R}}^{_{(2)}} + \pi_{r}\cE_{R}$.
    Using the same line of argument, we can establish the upper bound $J_2 \lesssim \pi^{_{(2)}}_{^{r, R}} + \pi_{r}\cE_{R}$,
    proving the bound on $J_2$. 
    \medskip
    
        It remains to prove the bound for $J_3$, for which we work on the domain $u^{-\gamma}h(R/2) \leq 1$, $\pi_{(u),r}\leq 1/2$.
     {Recall the definition of $J_3$ and apply} \eqref{eq:assmpt_theta_u_R} to the integrand to obtain 
    \begin{align*}
        J_3 & \lesssim  \int_{B(0, R/2)\setminus B(0,4r)}\int_{0}^1   \mathbb P \big( (0,u) \sim (y,v) \big)( \pi^{(2)}_{(v), R/2} + \cE_{R/2}) dvdy.
    \end{align*}
    Expanding the sum in the integrand, the right-hand side can be upper bounded by a constant multiple of 
    \begin{align}
        & \int_{B(0, R/2)\setminus B(0,4r)}\int_{0}^1   \mathbb P \big( (0,u) \sim (y,v) \big) \pi^{(2)}_{(v), R/2} dvdy\notag \\
        & + \cE_{{R}} \int_{B(0, R/2)\setminus B(0,4r)}\int_{0}^1   \mathbb P \big( (0,u) \sim (y,v) \big) dvdy,\label{eq:J_3_proof_bound}
    \end{align}
    using the assumption that $\cE_{R/2}\lesssim \cE_R$. To study the second summand, recall that the number of edges from $(0,u)$ to vertices in $B(0,r)^\complement$ is Poisson distributed, with expectation
    \[\lambda  \int_{B(0,r)^\complement}\int_{0}^1   \mathbb P \big( (0,u) \sim (y,v) \big) dvdy.
    \]
    Its probability of being positive is $\pi_{(u),r}$, which is assumed to be less than $1/2$, and therefore we can bound the integral above by some constant multiple of $\pi_{(u),r}$. This implies that the second summand in \eqref{eq:J_3_proof_bound} is bounded by a constant multiple of \smash{$\pi_{(u),r} \cE_{{R}}$}.
\smallskip

For the first summand of \eqref{eq:J_3_proof_bound}, we recall the function $h$ from Table~\ref{tab:gandh}. Lemma~\ref{lem:integral_edge_crossing} \eqref{eq:integral_edge_crossing_ball} states that the first summand can be bounded by 
\begin{align*}
        \int_{B(0, R/2)\setminus B(0,4r)}\int_{0}^1   \mathbb P \big( (0,u) \sim (y,v) \big) \pi^{(2)}_{(v), R/2} dvdy \lesssim \pi_{(u), R/2}^{(2)} \quad
        \text{ on $\{u^{-\gamma}h(R/2) \leq 1\}$. }
    \end{align*}
    The bound for $J_3$ follows from the bounds obtained for both terms in \eqref{eq:J_3_proof_bound}.
\end{proof}

\subsection{The bootstrap argument} \label{subsec:bootstrap_argument}
In Section~\ref{subsec:naive_upper_bound} we have shown that
\begin{align*}
    \theta_{r, R}  &\asymp  \pi_{r, R}^{(2)}  \text{ on }  D_0,\\
    \theta_{(u), R}  &\asymp  \pi_{(u), R}^{(2)}  \text{ on } \widetilde D_0,
\end{align*}
where $D_0$ and $\widetilde D_0$ are defined in \eqref{eq:naive_Dzero}, respectively~\eqref{def:set_tD1}. The idea behind the bootstrap argument lies in recursively applying the geometric argument given in Proposition~\ref{prop:bootstrapping} to show that these bounds also hold on increasingly larger domains $D_1, \widetilde D_1$; $D_2, \widetilde D_2$; $\ldots$. 
We then show in Lemma~\ref{lem:ksuffices} that, given $\epsilon>0$ as in the statement of Theorem~\ref{thm:two-edges}, we can choose a large (but finite) $k$ such that $D_k$ contains $\{r \ge R^\epsilon\}$ and $\widetilde D_k$ contains $\{u\leq R^{-\epsilon}\}$, proving Theorem~\ref{thm:two-edges}. Our argument relies on obtaining bounds for $\theta_{r,R}$ and $\theta_{(u),R}$ simultaneously. 
\medskip

Note that the bounds obtained for $\theta_{(u),R}$ in Proposition~\ref{prop:bootstrapping} 
require choosing the radius~$r$ appropriately within the range $[2,R/32]$. It is natural to choose a value $s=s(u,R)$ such that 
$\pi_{^{s,R}}^{_{(2)}}$ and $\pi_{^{(u),R}}^{_{(2)}}$ are comparable. Given that we eventually prove Theorem~\ref{thm:two-edges}, this can heuristically be viewed as requiring the event of escaping the ball $B(0,R)$ to be as likely when starting a path from a single (possibly strong) vertex $(0,u)$ as from a ball $B(0,s(u,R)).$ We thus define
\begin{equation}\label{eq:definition_s(u,R)}
s(u,R):= \inf\{s\in [2,R/32],\ \pi_{s,R}^{(2)} \ge \pi^{(2)}_{(u),R}\},
\end{equation}
with the convention $\inf \emptyset = R/32$. Because the infimum $s(u,R)$ may be attained  at boundary points of the interval $[2,R/32]$ we cannot guarantee that $\pi_{^{s,R}}^{_{(2)}}=\pi^{_{(2)}}_{^{(u),R}}$. However, Lemma~\ref{lem:utos} below shows that the quantities are always comparable and gives the asymptotic value of the function $s(u,R)$. 
Before stating this lemma, we
introduce the following notation. Recall~$\cD$ from \eqref{def:domain_D}. For a domain $D \subseteq \cD$, we write
\begin{align}\label{def:usll}
    f(u,r,R)\usll g(u,r,R) \text{ on }D,
\end{align}
if there exists
$a\ge 0$ such that 
$$f(u,r,R)\lesssim g(u,r,R) (\log R)^a \text{ on $D $,}$$  
and  
$f(u,r,R)\usl g(u,r,R)$ on $D$
if additionally $g(u,r,R) \,\smash{\usll}\, f(u,r,R)$ on $D$. 
\smallskip

\noindent
Note that  if $f_1 \,\smash{\usll}\, f_2$ and $g_1 \,\smash{\usll} \,g_2$ then $f_1 g_1 \,\smash{\usll} \,f_2 g_2$.
\medskip

\begin{lemma}\label{lem:utos} 
The function $s\colon (0,1]\to [2,R/32], u\mapsto s(u,R)$ 
is continuous. 
It satisfies
\begin{align}
    \pi_{s(u,R),R}^{(2)}&\asymp \pi_{(u),R}^{(2)},\label{asymp:s(u,R)_2e_crossing}\\
    \pi_{(u),s(u,R)}&\lesssim \pi_{s(u,R)} \label{asymp:s(u,R)_1e_crossing}
\end{align}
on the domain $\widetilde D_0^\complement$, 
where it also satisfies the following asymptotics,
\begin{equation} \label{eq:asymptotics_for_s(u,R)}
\begin{cases}
    s(u,R) \underset{^{{\log}}}\asymp u^{-1/d} &\text{ if }\delta\ge 1/\gamma,\\
      s(u,R) \asymp u^{-\gamma \delta/d} &\text{ if }\delta<1/\gamma.
\end{cases}
\end{equation}
\end{lemma}


\begin{proof}
    The continuity statement follows directly from continuity and  strict monotonicity of \smash{$u\mapsto \pi_{^{(u),R}}^{_{(2)}}$} and  \smash{$r\mapsto \pi_{^{r,R}}^{_{(2)}}$}, see Lemma~\ref{lem:monotone} and Remark~\ref{rem:continuity_monotonicity}. This continuity argument also yields \smash{$\pi_{^{s(u,R),R}}^{_{(2)}}= \pi_{^{(u),R}}^{_{(2)}}$} when $s(u,R)$ is in the open interval $(2,R/32)$.
    On the domain $\{s(u,R)=2\}$, we have by \smash{$\pi_{^{2,R}}^{_{(2)}}\ge \pi_{^{(u),R}}^{_{(2)}}$} by definition of $s(u,R)$. Using the monotonicity of $\pi_{^{(u),R}}^{_{(2)}}$ and the estimates in \eqref{asymp:2e-escape-probabilities} and \eqref{asymp:2e-crossing-probabilites}, we also obtain \smash{$\pi_{^{(u),R}}^{_{(2)}}\ge \pi_{^{(1),R}}^{_{(2)}}\asymp \pi_{^{2,R}}^{_{(2)}}$}, so we conclude
    \smash{$\pi_{_{s(u,R),R}}^{_{(2)}} \asymp \pi_{_{(u),R}}^{_{(2)}}$}.      
    Similarly, on \smash{$\{s(u,R)=R/32\}\cap \widetilde D_0^\complement,$} we have both \smash{$\pi_{^{s(u,R),R}}^{_{(2)}}\le \pi_{^{(u),R}}^{_{(2)}}$} by definition of $s(u,R)$ in~\eqref{eq:definition_s(u,R)}, 
\smash{$\pi_{^{s(u,R),R}}^{_{(2)}}=\pi_{^{R/32,R}}^{_{(2)}}\asymp \pi_R$} by~\eqref{eq:annulus_r_2r_crossing_with_one_or_two_edges} and Lemma~\ref{lem:renormalization_constants},
    and \smash{$\pi_R\gtrsim \pi_{^{(u),R}}^{_{(2)}}$} by~\eqref{ineq:defining_tildeD0_complement}. Combining our results, we obtain~\eqref{asymp:s(u,R)_2e_crossing} on this domain.
    \smallskip

We now prove \eqref{asymp:s(u,R)_1e_crossing} and \eqref{eq:asymptotics_for_s(u,R)} simultaneously by using~\eqref{asymp:s(u,R)_2e_crossing} together with the asymp\-totic expressions obtained in Proposition~\ref{prop:asymp_escape_crossing_probabilities} for the one edge and two edge crossing probabilities. We consider the different cases for $\gamma$ and $\delta$.\smallskip

\noindent
$\bullet$ Suppose first that $2<\delta<\frac{1}{\gamma}-1$, so that $\pi_{^{s(u,R),R}}^{_{(2)}}\asymp\pi_{s(u,R),R}$ and $\pi_{^{(u),R}}^{_{(2)}}\asymp\pi_{(u),R}$. In this case,  \eqref{asymp:1e-escape-probabilities} and \eqref{asymp:1e-crossing-probabilites} together with~\eqref{asymp:s(u,R)_2e_crossing} gives
\[s(u,R)^{d}g(R)\asymp u^{-\gamma \delta}g(R).\]
Dividing by $g(R)$ gives $s(u,R)\asymp u^{-\gamma\delta/d}$, and multiplying by $g(s(u,R))$ gives $ \pi_{(u),s(u,R)}\asymp \pi_{s(u,R)}$, using again \eqref{asymp:1e-escape-probabilities} and \eqref{asymp:1e-crossing-probabilites}, thus proving \eqref{asymp:s(u,R)_1e_crossing} and \eqref{eq:asymptotics_for_s(u,R)}.
\smallskip

$\bullet$ The case \smash{$\frac{1}{\gamma}-1\leq \delta<\frac{1}{\gamma}$} is only slightly more difficult, since combining~\eqref{asymp:s(u,R)_2e_crossing} together with the asymptotic expressions obtained in Proposition~\ref{prop:asymp_escape_crossing_probabilities} yields
\[s(u,R)^dg(R)+s(u,R)^{\frac{d}{\delta}}h(R)\asymp u^{-\gamma\delta}g(R)+u^{-\gamma}h(R).\]
In particular there is some $C>0$ such that for all $R$ large,
\begin{equation}\label{eq: asymptotic_sum}
s(u,R)^dg(R)+s(u,R)^{\frac{d}{\delta}}h(R)\ge C\big(u^{-\gamma\delta}g(R)+u^{-\gamma}h(R)\big),\end{equation} so either $s(u,R)^dg(R)\geq Cu^{-\gamma\delta}g(R)$ or $s(u,R)^{\frac{d}{\delta}}h(R)\ge Cu^{-\gamma}h(R)$ holds, depending on which term on the left hand side of \eqref{eq: asymptotic_sum} is bigger. In both cases we arrive at $s(u,R)\ge Cu^{-\gamma\delta/d}$ from which we conclude that $s(u,R)\lesssim u^{-\gamma\delta/d}$. Following the same steps we can prove that $s(u,R)\gtrsim u^{-\gamma\delta/d}$ and hence $s(u,R)\asymp u^{-\gamma\delta/d}$, which already proves \eqref{eq:asymptotics_for_s(u,R)} and can be used to prove \eqref{asymp:s(u,R)_1e_crossing} as in the previous case.\medskip

\noindent
$\bullet$ The case \smash{$\delta=\frac{1}{\gamma}$} requires a slightly more delicate analysis. In this case~\eqref{asymp:s(u,R)_2e_crossing} together with the asymptotic expressions from Proposition~\ref{prop:asymp_escape_crossing_probabilities} gives, for $s=s(u,R)$,
\begin{equation}\label{eq: asymptotic_sum_3}s^d\log(R)g(R)+s^{d\gamma}(\log(s))^{\gamma}h(R)\asymp u^{-1}g(R)+u^{-\gamma}h(R).\end{equation}
 We claim that there is $C'>0$ such that
$s^{d}\log(s)\ge C' u^{-1}$.
To see this, for every $R$ large we have two possible scenarios. First, if $s^{d\gamma}(\log(s)^{\gamma}h(R)\ge s^d\log(R)g(R)$, then we deduce
 from \eqref{eq: asymptotic_sum_3} that there exists $C>1$ such that for all large $R$ we have $2s^{d\gamma}(\log(s))^{\gamma}\ge C u^{-\gamma}h(R)$ from which the claim follows. Second, if $s^d\log(R)g(R)\ge s^{d\gamma}(\log(s)^{\gamma}h(R)$ we plug the asymptotic expressions for $g(R)$ and $h(R)$ from Table~\ref{tab:gandh} in this inequality and get $s\geq C'R/(\log(R))^{\beta}$ for some $C'>0$ and $\beta>0$ independent of $R$. In particular, for $R$ large  we have $\log(R)\leq 2\beta\log(s)$ and hence from \eqref{eq: asymptotic_sum_3} we deduce $4\beta s^{d}\log(s)\ge C u^{-1}$, which again gives the claim.\smallskip
 
 From $s^{d}\log(s)\ge C' u^{-1}$ it  directly follows that $u^{-1/d}$ \smash{\tiny\raisebox{3pt}{$\usll$}} $s(u,R)$, thus giving the correct lower bounds for \eqref{eq:asymptotics_for_s(u,R)}. Also, multiplying the inequality by $g(s)$ and using the asymptotic expressions in \eqref{asymp:1e-escape-probabilities} and \eqref{asymp:1e-crossing-probabilites} gives \eqref{asymp:s(u,R)_1e_crossing}. Finally, observe that \eqref{eq: asymptotic_sum_3} together with the bound $\log(R)\geq\log(s)$ also implies the existence of a constant $C>0$ such that
 \[C\big(s^d\log(s)g(R)+s^{d\gamma}(\log(s)^{\gamma}h(R)\big)\le u^{-1}g(R)+u^{-\gamma}h(R),\]
 from which $s(u,R)\,\smash{\usll}\, u^{-1/d}$ follows using the same steps as in the case  $\frac{1}{\gamma}-1\leq \delta<\frac{1}{\gamma}$.\medskip

\noindent 
$\bullet$ In the case $\delta>\frac{1}{\gamma}$, the asymptotic equality~\eqref{asymp:s(u,R)_2e_crossing} gives
\[s^d\log(R)g(R)+s^{d\gamma}h(R)\asymp u^{-1}g(R)+u^{-\gamma}h(R).\]
The proof in this case is analogous to the one with $\delta=\frac{1}{\gamma}$ so we omit the details.\end{proof}

We are now ready to recursively increase the domain on which our main result holds, {see Figure~\ref{fig:bootstrap} for an illustration. For every $R\ge4$ we define a decreasing sequence $(r_k(R))_{k\ge 0}$},~by
\begin{align}
\begin{split}
    r_0(R)&:=R/32, \\
    r_{k+1}(R)&:= \inf\left\{r \in [2, r_k(R)],\  \pi_{r,R}^{(2)}\ge \pi_{r}\ \pi_{r_{k}(R),R}^{(2)}\right\} \quad \text{ for any } k\ge 0. 
\end{split}    
\label{eq:radius_r_kplus1}
\end{align}
\begin{figure}[h]
    \centering
    \includegraphics[width=10cm]{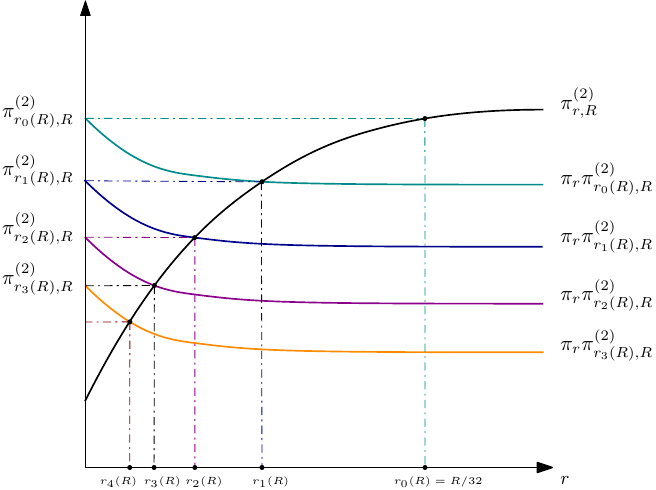}
    \caption{An illustration how the geometric argument
     will be used. Successive application of Proposition~4.3 gives bounds which are the maximum of the black and the coloured lines, so that the domain for \smash{$\theta_{r,R} \lesssim \pi^{(2)}_{r,R}$} after $k$ applications is the interval to the right of the intersection points $r_k(R)$. }
    \label{fig:bootstrap}
\end{figure}
\begin{lemma}\label{lem:recursive_def_rk}
    For any $k\ge 0,$ the following holds,
    \begin{align}
    &r_k(R)\underset{R\to \infty} \longrightarrow \infty.\label{eq:rk_div} \\[0.8ex]
    &\pi_{^{r_{k+1}(R),R}}^{_{(2)}}= \pi_{r_{k+1}(R)} \pi_{^{r_{k}(R),R}}^{_{(2)}}\;\text{ for all large $R$}\label{eq:recursive_def_rk}.\\[1.2ex]
    &r_k(R)\lesssim r_{k}(2R)\label{eq:rk_comparable}.
\end{align}\end{lemma}

\begin{proof}
    We begin by showing \eqref{eq:rk_div} by induction. The case $k=0$ follows immediately from the definition of $r_0(R) = R/32$. Next, suppose that $r_k(R)$ tends to infinity as $R\uparrow\infty$. We claim that for any fixed $r>2$, 
    \begin{align}\label{eq:pi_rk_order_limit}
        \lim_{R\to \infty} \frac{\pi_{r,R}^{(2)}}{\pi_{r_k(R),R}^{(2)}} = 0.
    \end{align}
    {We write out the proof only for the case  $\delta > 1/\gamma$ as the other cases can be shown similarly.} 
    The bounds \eqref{asymp:2e-crossing-probabilites} and \eqref{asymp:1e-crossing-probabilites} on $\pi_{^{r,R}}^{_{(2)}}$ and $\pi_{r,R}$ yield 
    \begin{align*}
        \pi_{r,R}^{(2)} & \asymp r^d\log(R) g(R) + r^{d\gamma}h(R), \text{ and }\\
        \pi_{r_k(R),R}^{(2)} & \asymp (r_k(R))^d\log(R) g(R) + (r_k(R))^{d\gamma}h(R).
    \end{align*} 
    Thus, \smash{$  \pi_{^{r,R}}^{_{(2)}}/  \pi_{^{r_k(R),R}}^{_{(2)}} \lesssim (r/r_k(R))^d + (r/r_k(R))^{d\gamma}$}. Using the assumption that $r_k(R)$ tends to infinity and $r$ is fixed, the desired limit \eqref{eq:pi_rk_order_limit} follows.
    Hence, for large $R$ {and fixed $r$}, we have
    \smash{$   \pi_{^{r,R}}^{_{(2)}}< \pi_{r} \pi_{^{r_k(R),R}}^{_{(2)}}.$} 
    By definition of $r_{k+1}$ in~\eqref{eq:radius_r_kplus1} we therefore have that $r\le r_{k+1}(R)$ for all $R$ large. Since this holds for all fixed $r>2$, this implies that $r_{k+1}(R)\to \infty$ as~$R \to \infty.$
    \smallskip\pagebreak[3]

    {The equality \eqref{eq:recursive_def_rk} follows readily from continuity in $r$ of $\pi_r$ and $\pi_{^{r,R}}^{_{(2)}},$ once we ensure that $r_{k+1}(R)$ is in the open interval $(2,r_k(R))$.}
    This follows from the fact that it tends to infinity, and that $r_{k+1}(R)<r_k(R)$,
    which is easy to see. \smallskip

    {Finally, we show \eqref{eq:rk_comparable} by induction, with the case $k=0$ following immediately from the definition of $r_0(R) = R/32$. Suppose now that $r_k(R)\lesssim r_k(2R)$, and let $C>0$ be a large  constant that we determine later. By \eqref{eq:rk_div} and \eqref{eq:recursive_def_rk}, we can find a large constant~$R_0$ depending on $C$, such that for all $R\geq R_0$ we have \smash{$\pi_{^{r_{k+1}(2R),2R}}^{_{(2)}}= \pi_{r_{k+1}(2R)} \pi_{^{r_{k}(2R),2R}}^{_{(2)}}$} and $r_{k+1}(2R)\geq C$. 
    For such $R$ we prove \eqref{eq:rk_comparable} considering only the case $\delta > 1/\gamma$, as the remaining cases are analogous. Let $r'=Cr_{k+1}(2R)$. By \eqref{asymp:1e-crossing-probabilites}, \eqref{asymp:2e-crossing-probabilites} and Lemma~\ref{lem:renormalization_constants}, there exist constants $c_1$, $c_2$ such that
    \begin{align*}\frac{\pi_{^{r',R}}^{_{(2)}}}{\pi_{r'}}
    & \ge c_1\frac{C^{d}(r_{{k+1}}(2R))^d\log(R)g(R)+C^{d\gamma}(r_{{k+1}}(2R))^{d\gamma}h(R)}{C^{d(2-\frac{1}{\gamma})}(r_{{k+1}}(2R))^{d(2-\frac{1}{\gamma})}\log(r_{{k+1}}(2R))} \\ & \geq c_1c_2C^{2d(\frac{1}{\gamma}-1)}\frac{\pi_{^{r_{k+1}(2R),2R}}^{_{(2)}}}{\pi_{r_{k+1}(2R)}},\end{align*}
    where we used $Cr_{k+1}(2R)\leq (r_{k+1}(2R))^2$ in the first inequality, to handle the logarithmic term. Moreover the induction hypothesis gives a constant $c>0$ such that $r_k(2R)\geq c r_k(R)$, so using that $R\geq R_0$ and Lemma~\ref{lem:renormalization_constants}, there is a constant $c_3>0$ such that
    \[\frac{\pi_{^{r_{k+1}(2R),2R}}^{_{(2)}}}{\pi_{r_{k+1}(2R)}}=\pi_{^{r_{k}(2R),2R}}^{_{(2)}}\geq \pi_{^{cr_{k}(R),2R}}^{_{(2)}}\geq c_3\pi_{^{r_{k}(R),R}}^{_{(2)}}.\]
    Combining this inequality with the previous one we obtain, for sufficiently large $C$, 
    \[\frac{\pi_{^{r',R}}^{_{(2)}}}{\pi_{r'}}\ge c_1c_2c_3C^{2d(\frac{1}{\gamma}-1)}\pi_{^{r_{k}(R),R}}^{_{(2)}}\geq \pi_{^{r_{k}(R),R}}^{_{(2)}}.\]
    Hence $r'$ satisfies \smash{$\pi_{^{r',R}}^{_{(2)}}\ge \pi_{r'}\pi_{^{r_{k}(R),R}}^{_{(2)}}$}. Now, if $Cr_{k+1}(2R)\le R/32$, then by definition of $r_{k+1}(R)$ we have $r_{k+1}(R)\leq Cr_{k+1}(2R)$, while if $Cr_{k+1}(2R)\ge R/32$ the same inequality holds since $r_{k+1}(R)\leq R/32$. In either case, we conclude \eqref{eq:rk_comparable} for $k+1$.}\end{proof}

{We now introduce two sequences of increasing domains $(D_k)_{k\geq 0}$ and $(\widetilde D_k)_{k\geq 0}$, and state 
in the next proposition 
that our main result, Theorem \ref{thm:two-edges}, holds on these domains. Define}
\begin{align}
D_k&:=\{r_k(R)\le r \le R/2\}, \label{def:domain_D_k}\\
\widetilde D_k &:= \{s(u,R)\ge r_k(R)\}\, \cup \, \widetilde D_0,\label{def:domain_D_k_tilde}
\end{align}
for $k\geq 1$ and recall $D_0$ and $\widetilde D_0$ from \eqref{eq:def_Dzero} and \eqref{def:set_tD1}. It follows directly from the definition that the sequences $(D_k)_{k\geq 0}$ and \smash{$(\widetilde D_k)_{k\geq 0}$} are increasing. 
\pagebreak[3]

\begin{proposition} \label{prop:induction_on_domain}
    Let $(r_k(R))_{k\geq 0}$ be the sequence of functions defined by \eqref{eq:radius_r_kplus1}. For any $k\ge 0$, we have that
      \begin{align*}
      \theta_{r, R} & \asymp \pi_{r, R}^{(2)}
    \quad  \text{ on } D_{k},
       \\[2mm]
      \theta_{(u), R} & \asymp 
                 \pi_{(u), R}^{(2)} 
                 \quad \text{ on } \widetilde D_{k}.
  \end{align*}  
\end{proposition}

   \begin{proof}
   The proof of this proposition proceeds by induction: specifically we show that the bounds $\theta_{r,R}\asymp \pi_{^{r,R}}^{_{(2)}}$ on $D_k$ and $\theta_{(u),R}\asymp \pi_{^{(u),R}}^{_{(2)}}$ on~\smash{$\widetilde D_k$} hold jointly at each step $k$. The case $k=0$ is proved in \eqref{eq:naive_Dzero} and in Lemma~\ref{lem:naive} \eqref{eq:naive_u}. For the inductive step, fix an integer $k\ge0$ and suppose that the asymptotic equivalences $\theta_{r,R}\asymp \pi_{^{r,R}}^{_{(2)}}$ on $D_k$ and $\theta_{(u),R}\asymp \pi_{^{(u),R}}^{_{(2)}}$ on~\smash{$\widetilde D_k$} hold, so it remains to prove that these hold on $D_{k+1}\setminus D_k$ and \smash{$\widetilde{D}_{k+1}\setminus\widetilde{D}_k$}, respectively. Define $\cE_R=\pi_{^{r_k(R),R}}^{_{(2)}}$, and observe that Lemma~\ref{lem:recursive_def_rk} \eqref{eq:rk_comparable}, together with monotonicity and Lemma~\ref{lem:renormalization_constants}, gives some constant $c>0$ such that
   \[\cE_{2R}=\pi_{^{r_k(2R),2R}}^{_{(2)}}\geq \pi_{^{cr_k(R),2R}}^{_{(2)}}\asymp \pi_{^{r_k(R),R}}^{_{(2)}}=\cE_R,\]
   thus proving that $\cE_R\lesssim \cE_{2R}$. Our aim is to apply Proposition~\ref{prop:bootstrapping}, hence we first focus on showing that \eqref{eq:assmpt_theta_r_R} and \eqref{eq:assmpt_theta_u_R} are satisfied  with this choice of $\cE_R$. Note that our induction hypothesis already gives \eqref{eq:assmpt_theta_r_R} and \eqref{eq:assmpt_theta_u_R} on $D_k$ and $\widetilde D_k$, respectively. 
   We now prove \eqref{eq:assmpt_theta_r_R} on $D_k^{\complement}$, where   we use monotonicity of $\theta_{r,R}$ in $r$, given in Lemma~\ref{lem:monotone}, to obtain
    \begin{align*}
         \theta_{r,R}& \le \theta_{r_k(R),R}\asymp \pi_{r_k(R),R}^{(2)}\, \quad {\text{ on }D_k^{\complement}},
    \end{align*}
    where the asymptotic equivalence holds by the induction hypothesis. Finally, we prove that \eqref{eq:assmpt_theta_u_R} holds on $\widetilde{D}_k^{\complement}$. Define
    \begin{equation}\label{eq:def_uk}
         u_k(R) =\sup\{u\in (0,1]\colon s(u,R)\ge r_k(R)\}.
    \end{equation}
    Recalling that $r_k(R)\to \infty$ as $R\to \infty$ by Lemma~\ref{lem:recursive_def_rk} and~\eqref{eq:asymptotics_for_s(u,R)}, we deduce $u_k(R)\to 0$ as $R\to \infty$. By continuity of $s(u,R)$, given in Lemma~\ref{lem:utos}, we have $s(u_k(R),R)=r_k(R)$ provided $R\geq R_0$ with sufficiently large $R_0$, and hence 
    \begin{equation}\label{eq:equiv_tilDk}
    \widetilde D_k^{\complement}\cap\{R\geq R_0\}=\{u\ge u_k(R),\,R\ge R_0\}.
    \end{equation}
    We can use the monotonicity of $\theta_{(u),R}$ in $u$, given in Lemma~\ref{lem:monotone},  to obtain
    \begin{align*}
        \theta_{(u),R}& \le \theta_{(u_k(R)),R}\asymp \pi_{(u_k(R)),R}^{(2)}\asymp \pi_{s(u_k(R),R),R}^{(2)}=\pi_{r_k(R),R}^{(2)} \quad \text{ on }\widetilde D_k^{\complement},
    \end{align*}
    where the first asymptotic equivalence holds by the induction hypothesis, the second follows from Lemma~\ref{lem:utos} \eqref{asymp:s(u,R)_2e_crossing} and the fact that $\widetilde D_k^{\complement} \subseteq \widetilde D_0^{\complement}$, and the third from $s(u_k(R),R)=r_k(R)$. Combining this with the induction hypothesis, we see that~\eqref{eq:assmpt_theta_r_R} and~\eqref{eq:assmpt_theta_u_R} are satisfied with $\cE_R=\pi_{^{r_k(R),R}}^{_{(2)}}$, and hence we can apply  Proposition~\ref{prop:bootstrapping}.
    \smallskip
    
    \noindent $\blacktriangleright$ First, we prove $\theta_{r,R}\asymp \pi_{^{}{r,R}}^{_{(2)}}$ on $D_{k+1}\setminus D_k$. Proposition~\ref{prop:bootstrapping} yields
    \[
    \pi_{r,R}^{(2)}\leq\theta_{r,R}\lesssim \pi_{r,R}^{(2)}+ \pi_r\cE_{R}\quad\text{ on } D_{k+1}\setminus D_k,
    \]
    and since $r\geq r_{k+1}(R)$, we have from  \eqref{asymp:1e-crossing-probabilites} that $\pi_r\lesssim \pi_{r_{k+1}(R)}$.
    Hence, 
     \begin{align*} 
        \pi_r \cE_{R} &=\pi_r \pi_{r_k(R),R}^{_{(2)}} \lesssim \pi_{r_{k+1}(R)} \ \pi_{r_k(R),R}^{_{(2)}} = \pi_{r_{k+1}(R), R}^{_{(2)}},
    \end{align*}
    where the last equality holds by Lemma~\ref{lem:recursive_def_rk} for large enough $R$. Using the fact that $\pi_{^{r,R}}^{_{(2)}}$ is {increasing} in $r$, see Lemma~\ref{lem:monotone}, it follows that  \smash{$ \pi_{^{r_{k+1}(R), R}}^{_{(2)}} \le \pi_{r , R}^{_{(2)}}$}    and so we obtain 
    \smash{$\theta_{r,R} \asymp \pi_{^{r,R}}^{_{(2)}}$} on
   $D_{k+1}$, as required.\smallskip
   
    \noindent $\blacktriangleright$ 
Repeating the previous argument, it remains to prove $\theta_{(u),R}\asymp \pi_{^{(u),R}}^{_{(2)}}$ on $\widetilde D_{k+1}\setminus \widetilde D_k$. On this domain, Proposition~\ref{prop:bootstrapping} yields that
    \[
    \theta_{u,R}\lesssim \pi_{(u),R}^{_{(2)}} + \pi_{s,R}^{_{(2)}} + (\pi_s + \pi_{(u),s}) \cE_{R},
    \]
    where $s$ is arbitrary in $[2,R/32]$. We choose $s=s(u,R)$, which by Lemma~\ref{lem:utos} yields
    \[
    \theta_{(u),R}\lesssim  \pi_{(u),R}^{_{(2)}}+ \pi_{s(u,R)} \cE_{R} \quad \text{ on }  \widetilde D_{k+1}\setminus \widetilde D_k,
    \]
    since by Lemma~\ref{lem:utos}, the above holds on $\widetilde D_0^{\complement}$, and we have $\widetilde D_0 \subseteq \widetilde D_k$, hence $\widetilde D_{k+1}\setminus \widetilde D_k \subseteq \widetilde D_0^{\complement}$.     We then argue as before, still with $s=s(u,R)$, that
     \begin{align*}
        \pi_s \cE_{R} &=\pi_s \pi_{r_k(R),R}^{_{(2)}} \lesssim \pi_{r_{k+1}(R)} \ \pi_{r_k(R),R}^{_{(2)}} 
         = \pi_{r_{k+1}(R),R}^{_{(2)}} 
         \le \pi_{s,R}^{_{(2)}}\asymp \pi_{(u),R}^{_{(2)}},
    \end{align*}
    whence
    $
    \theta_{(u),R}\asymp \pi_{(u),R}^{_{(2)}}$
    on $\widetilde D_{k+1},
    $
    and Proposition~\ref{prop:induction_on_domain} is proved.\qedhere
\end{proof}

\section{Proof of the theorems} 
\label{sec:thms}

\subsection{Proof of Theorem~\ref{thm:two-edges}.}
In order to complete the proof of Theorem~\ref{thm:two-edges}, we must verify that, for sufficiently large $k$, the domains $D_k$ and $\widetilde D_k$ contain the desired domains of the theorem. This is established in the following lemma. 

\begin{lemma}\label{lem:ksuffices}
    There exists $k\in \N$ and $R_0\,{> 4}$ such that
    \begin{align*}
        D_k \cap\{R>R_0\}&\supset \{ \tfrac R 2 \ge r \ge R^\epsilon \}\cap\{R>R_0\},\\
        \widetilde D_k \cap\{R>R_0\}&\supset \{R^{-\epsilon}\ge u\}\cap\{R>R_0\}.
    \end{align*}
\end{lemma}
\noindent
Theorem~\ref{thm:two-edges} follows by combining Proposition~\ref{prop:induction_on_domain}, Lemma~\ref{lem:ksuffices} choosing $k$ sufficiently large.
\begin{proof}[Proof of Lemma~\ref{lem:ksuffices}]
    We show that the sequence $(r_k(R))_{k\geq 0}$ defined by~\eqref{eq:radius_r_kplus1} satisfies
    \begin{align}\label{eq:rkalpha}
        r_k(R)\usl R^{\alpha_k},
    \end{align}
    where $(\alpha_k)_{k\geq 0}$ is a sequence of positive reals depending on $\gamma$ and $\delta$, that satisfies 
    $\alpha_0=1$ and $\alpha_k\to 0$, and where we recall the notation
    {$\usl$} from \eqref{def:usll}. \smallskip
    
    The result for $D_k\cap\{R\geq R_0\}$ then follows 
    from the above as \eqref{eq:rkalpha} implies that there exists constants $C>0$ and $a\geq 0$ such that, for sufficiently large~$R$ and $k$,
    \begin{equation}\label{eq: rk_larger}r_k(R) \geq C R^{\alpha_k}(\log R)^a \geq R^{\epsilon},\end{equation} since $\epsilon>0$ is fixed. To obtain the result for $\widetilde{D}_k\cap\{R\geq R_0\}$, we can deduce 
    \[\widetilde D_k\cap\{R\geq R_0\}=\{u\le u_k(R),\,R\ge R_0\}\]
    analogously to \eqref{eq:equiv_tilDk}, where $u_k$ is defined in \eqref{eq:def_uk}, and satisfies $s(u_k(R),R)=r_k(R)$ for $R$ large. The result then follows from Lemma~\ref{lem:utos} \eqref{eq:asymptotics_for_s(u,R)} together with \eqref{eq: rk_larger}. We divide the proof of \eqref{eq:rkalpha} according to the values of the model parameters $\gamma$ and $\delta$, moving from the easier to the more involved cases.\smallskip

    \noindent$\bullet$ {\textbf{ Case $2<\delta\le \tfrac 1 \gamma-1:$}}
    Consider first the case $\delta<\tfrac 1 \gamma-1$ and recall from Proposition~\ref{prop:asymp_escape_crossing_probabilities} the asymptotics
\begin{align}\label{asymp_two_edges_deltasmall} 
        \pi_{r}&\asymp r^{d(2-\delta)}, \qquad
        \pi_{r,R}^{_{(2)}} \asymp r^{d} R^{d(1-\delta)}.
    \end{align}
    Therefore, 
    \begin{align*}
        \pi_{r_{k+1}(R), R}^{_{(2)}} & \asymp r_{k+1}(R)^d R^{d(1-\delta)} \,\,\text{ and } \,\, \pi_{r_{k+1}(R)}\pi^{_{(2)}}_{r_k(R), R}\asymp r_{k+1}(R)^{d(2-\delta)}r_k(R)^d R^{d(1-\delta)}.
    \end{align*}
    Combining the above with Lemma~\ref{lem:recursive_def_rk}\,\eqref{eq:recursive_def_rk} yields $r_{k+1}(R) \asymp r_k(R)^{1/(\delta - 1)}$ and by recursively applying this, we obtain 
    $$r_k(R) \asymp r_0(R)^{1/(\delta -1)^k}\asymp R^{1/(\delta -1)^k}.$$
    Hence  {$\alpha_k= 1/ {(\delta-1)^k} \to 0.$}
    For  \smash{$\delta= \tfrac 1 \gamma-1$}, {recall from Corollary~\ref{cor:2e-probabilities-boundarycases} that} the asymptotics~\eqref{asymp_two_edges_deltasmall} stay valid on the domain
    \[
    \{(\log R)^{\frac \delta {\delta-1}} \leq r^d\},
    \]
    from which we deduce again the same asymptotics for the sequence $(r_k(R))_{k\geq 0}$.
\smallskip

    \noindent$\bullet$ {\textbf{ Case $\delta\ge \tfrac 1 \gamma:$}}
    The precise asymptotics for $\pi_{r}$ and $\pi_{^{r,R}}^{_{(2)}}$ in Proposition~\ref{prop:asymp_escape_crossing_probabilities} are more involved in this case. But ignoring precise log factors, we can summarize these as   
    \[\pi_{r} \usl r^{d(2-\frac{1}{\gamma})}, \qquad
    \pi_{r,R}^{_{(2)}} \usl r^{d\gamma} R^{d(2-\gamma-\tfrac 1 \gamma)}.
    \]
 Therefore, 
    \begin{align*}
        \pi_{r_{k+1}(R),R}^{_{(2)}} & \usl r_{k+1}(R)^{d\gamma}R^{d(2-\gamma - 1/\gamma)}, \text{ and}\\
        \pi_{r_{k+1}(R)} \pi_{r_k(R),R}^{_{(2)}} & \usl r_{k+1}(R)^{d(2-\frac{1}{\gamma})}r_k(R)^{d\gamma}R^{d(2-\gamma - \frac{1}{\gamma})}.
    \end{align*}
    Similarly to the first case, we combine the above with Lemma~\ref{lem:recursive_def_rk} to obtain 
    \begin{align*}
        r_{k+1}(R)\usl r_k^{\left(\gamma/(1-\gamma)\right)^2} \usl R^{\alpha_k},
    \text{ with $\alpha_k = \big(\tfrac{\gamma}{1-\gamma } \big)^{2k}$. }
    \end{align*}
    As $\gamma < 1/2$, we have \smash{$\frac{\gamma}{1-\gamma}<1$} and hence $\alpha_k \to 0$, concluding this case. \smallskip
    

    \noindent$\bullet$ {\textbf{ Case $2 \vee (\tfrac 1 \gamma-1)<\delta <\tfrac 1 \gamma:$}} 
    Proposition~\ref{prop:asymp_escape_crossing_probabilities} now provides the asymptotics
    \begin{align*}
        \nonumber \pi_{r}&\asymp r^{d(2-\delta)}\\
        \label{asymp_two_edges_deltasmall} \pi_{r,R}^{_{(2)}}&\asymp 
        \begin{cases}
                    r^{\frac d \delta} R^{d(1+\frac 1{\gamma \delta}-\frac 1 \gamma-\frac 1 {\delta})} & \text{ on } 
                    \{  r \leq R^{1+\delta-\frac 1 \gamma}\},\\
                    r^d R^{d(1-\delta)} & \text{ on } \{  R^{1+\delta-\frac 1 \gamma} \leq r \leq R/2\}.
                \end{cases}
    \end{align*}
    We use this to obtain recursively $r_k(R)\asymp R^{\alpha_k}$, however the recursive definition of the sequence $(\alpha_k)$ now differs when this exponent crosses the value \smash{$1+\delta-\frac 1 \gamma.$} More precisely,
    \begin{itemize}
        \item[$\blacktriangleright$] For $k=1, \ldots, k_0$, with \smash{$k_0:= \max \big\{k\ge 0\colon\ \frac 1 {(\delta-1)^k} \ge 1+ \delta-\frac 1 \gamma\big\},$} we have
   \smash{$\alpha_k= \frac 1 {(\delta-1)^k}.$}
        \item[$\blacktriangleright$] Then $\alpha_{k_0+1}\in (0,\ 1+ \delta-\frac 1 \gamma).$
        \item[$\blacktriangleright$] Finally for $k\ge k_0+1$,
        \[
        \alpha_{k+1}= \frac {\alpha_k}{(\delta-1)^2}= \frac {\alpha_{k_0+1}}{(\delta-1)^{2(k-k_0)}}.
        \]
    \end{itemize}
    Thus $\alpha_k$ tends to 0, concluding this case.
\end{proof}

\subsection{Proof of Theorems~\ref{thm:main_R_alpha_R} and~\ref{thm:main_u_R}.}
Theorems~\ref{thm:main_R_alpha_R} and~\ref{thm:main_u_R} 
follow from Theorem~\ref{thm:two-edges} and the expressions obtained for $\pi_{{r,R}}, \pi^{_{(2)}}_{^{r,R}}$
and $\pi_{{(u),R}}, \pi^{_{(2)}}_{^{(u),R}}$ in 
Proposition~\ref{prop:asymp_escape_crossing_probabilities} by identifying the domain in which the two-edge probabilities dominate the one-edge probabilities. 

\subsection{Proof of Theorem~\ref{thm:subcritical_one_arm_exponent}.}

We assume that $\mathbb P^*$ is given by adding the point $(0,U)$ to the vertex set $\mathcal P$. 
We use the computations already made in Section~\ref{sec:one_two_crossing_bounds} to deduce
\begin{align*}
    \mathbb P^*\big( (0,U) \simtwo B(0,R)^\complement \big)
    &= \int_0^1 \pi_{(u),R}^{_{(2)}} \de u \asymp \pi_{(1),R}^{_{(2)}}
    = h(R).
\end{align*}
We deduce  the lower bounds in Theorem~\ref{thm:subcritical_one_arm_exponent} from Table~\ref{tab:gandh}, but in fact this yields lower bounds with up-to-constants precision. We conjecture that these bounds are sharp.
\smallskip


For the upper bound we use 
\begin{align*}
\mathbb P^*\big( (0,U) \longleftrightarrow B(0,R)^\complement \big) &
\le \int_0^{R^{-\epsilon}} \mathbb P\big( (0,u) \leftrightarrow B(0,R)^\complement \big) \, du+ 
\mathbb P\big(  (0,R^{-\epsilon}) \leftrightarrow B(0,R)^\complement \big)\\
& \asymp 
\int_0^{R^{-\epsilon}} \pi^{_{(2)}}_{(u),R} \, du+ 
\pi^{_{(2)}}_{(R^{-\epsilon}),R}.
\end{align*}
Using that
$\int  \mathbb P ( (0,1) \sim (x,v)) \,  dx$
is uniformly bounded from below and
Lemma~\ref{lem:integral_edge_crossing} we see
\begin{align*}
\int_0^{R^{-\epsilon}} \pi_{(u),R}^{_{(2)}} \, du 
& \lesssim  \int_0^1 \int  \mathbb P \big( (0,1) \sim (x,v) \big) \, \de x\,\pi_{(v),R}^{_{(2)}} \, \de v  \lesssim \pi_{(1),R}^{_{(2)}} =h(R),
\end{align*}
which we get again from Table~\ref{tab:gandh}.
By Proposition~\ref{prop:asymp_escape_crossing_probabilities} we have
$$\pi^{_{(2)}}_{(R^{-\epsilon}), R}
\lesssim \pi_{(R^{-\epsilon}), R} + R^{\epsilon\gamma} h(R) \wedge 1,
\text{ and }
\pi_{(R^{-\epsilon}), R}  \lesssim R^\epsilon g(R),$$
so that the upper bound follows from Table~\ref{tab:gandh} 
by choosing a small $\epsilon>0$.\bigskip

\emph{Acknowledgment:}
AL was supported by the ANID/FONDECYT regular grant 1252012. EJ was supported by grant GrHyDy ANR-20-CE40-0002.
CK and PM are supported by DFG project 444092244 ``Condensation in random geometric graphs" within the priority programme SPP~2265. This material is further based upon work supported by the National Science Foundation under Grant No. DMS-1928930, while CK and PM were in residence at the Simons Laufer Mathematical Sciences Institute in Berkeley, California, during the spring semester of 2025. 

\bibliographystyle{alpha}
\bibliography{citation}
\end{document}